\documentclass{article}
\usepackage{a4wide}
\usepackage[mac]{inputenc}
\usepackage{amssymb}
\usepackage{amsmath}
\usepackage{amsthm}
\usepackage{graphicx}

\newcommand{\N}{\mathbb{N}}
\newcommand{\I}{\mathbb{I}}
\newcommand{\R}{\mathbb{R}}

\newcommand{\dv}{\operatorname{div}}

\newcommand{\rf}{\rho_f}
\newcommand{\rs}{\rho_s}
\newcommand{\om}{\Omega}
\newcommand{\os}{{\Omega_s(t)}}

\newcommand{\oes}{{\Omega^\eta_{s}(t)}}
\newcommand{\oesb}{{\overline{\Omega^\eta_{s}}(t)}}
\newcommand{\of}{{\Omega_f(t)}}

\newcommand{\dom}{\partial\Omega}
\newcommand{\dos}{{\partial\Omega_s(t)}}

\newcommand{\ou}{{\omega_u(t)}}
\newcommand{\xg}{x_G(t)}
\renewcommand{\div}{\operatorname{div}}
\newcommand{\ds}{\displaystyle}
\newcommand{\ue}{u_\eta}
\newcommand{\Je}{J_\eta}
\newcommand{\Me}{M_\eta}
\newcommand{\ues}{u_{\eta,s}}
\newcommand{\re}{\rho_\eta}
\newcommand{\He}{H_\eta}
\newcommand{\pe}{p_\eta}
\newcommand{\eps}{\varepsilon}
\def\bmu{\begin{multline*}}
\def\emu{\end{multline*}}
\def\be{\begin{equation}}
\def\ee{\end{equation}}

\newtheorem{theorem}{Theorem}[section]
\newtheorem{lemma}[theorem]{Lemma}

\title{Numerical analysis of a penalization method for the three-dimensional motion of a rigid body in an incompressible viscous fluid} 
\author{C. Bost, G.-H. Cottet and E. Maitre\\
Universit\'e de Grenoble and CNRS, Laboratoire Jean Kuntzmann\\
BP 53, 38041 Grenoble Cedex 9, France}

\begin{document}

\maketitle
\begin{abstract}
We present and analyze a penalization method wich extends the the method of \cite{AngBruFab99} to the case of a rigid body moving freely in an incompressible fluid. The fluid-solid system is viewed as a single variable density flow with an interface captured by a level set method. The solid velocity is computed by averaging at avery time the flow velocity in the solid phase. This velocity is used to penalize the flow velocity at the fluid-solid interface and to move the interface. Numerical illustrations are provided to illustrate our convergence result. A discussion of our result in the light of existing existence results is also given.
\end{abstract}
\section{Introduction}
In this paper we are concerned with the numerical analysis of a penalization method for the two-way interaction of a rigid body with an incompressible fluid in three dimensions. The traditional numerical approach to deal with fluid-structure problems is the so-called ALE (for Arbitrary Lagrangian Eulerian) method where fluids (resp solids) are described in an Eulerian (resp Lagrangian) framework. Fluids are computed on a moving mesh fitting the solids and stress and velocity continuity are used to derive the appropriate boundary conditions on the fluid/solid interface. A convergence proof of a finite-element method based on this approach can be found in \cite{tuta}. 

Alternate methods can be devised where the whole fluid-solid system is seen as a multiphase flow and the fluid/solid interface is captured implicitly rather than explicitly. Likewise, the interface continuity conditions are recovered in an implicit fashion. The rigid motion inside the solid phase can be enforced through a Lagrange multiplier \cite{GloPan01}. The method we consider here is of this type but the rigid motion is approximately satisfied in the solid through penalization. 

Penalization methods have  already  been considered in the past for this problem.  In \cite{ConSanTuc00} the authors considered a single solid ball and worked inside the frame moving with its center. Then they penalized the mean velocity of the (virtual) fluid inside this ball. Their method  is restricted to one ball. In\cite{SanMarStaTuc02,maury} the penalization is applied to the deformation tensor inside the body. In \cite{SanMarStaTuc02} this method is used to prove the existence of solutions  for the fluid-solid interaction variational problem in two dimensions. In \cite{maury} it is used together with a two-dimensional finite element method in a variational framework. Here the penalization is applied to the flow velocity itself.  The method thus extends the one devised and analyzed in \cite{AngBruFab99} in the case of a rigid solid with prescribed motion.

In our method the determination of the body velocity is part of the problem. This velocity, instead of the flow velocity, is used to move the solid phase. This has a crucial practical importance, in particular for problems with large displacements and strong shear, since it ensures that the solid remains rigid at the discrete level, although the rigidity constraint in the flow field is only approximately satisfied. A vorticity formulation of the method and its validation on a number of 2D and 3D reference cases are given in \cite{Coquerelle2008}. 
An outline  of the paper is as follows. In section 2 we recall the weak formulation of the problem and we describe the penalization method. Section 3 is devoted to the convergence proof. In section 4 we provide some numerical illustrations.  Section 5 is devoted to some concluding remarks. The proofs of some technical results used in section 3 are given in the appendix. 

\section{Weak formulation and penalized problem}
\label{sec2}
Let $\om$ be an open bounded domain of $\R^3$, filled with a viscous incompressible and homogeneous fluid of density $\rf>0$ and viscosity $\mu>0$. Inside this domain, we consider the motion of an immersed homogeneous rigid solid of density $\rs>0$ during a time interval $[0,T]$, $T>0$, chosen so that  the solid never comes in contact with $\dom$.
For $t\in[0,T]$, we denote  by $\of$ and $\os$ the non-empty fluid and solid open connected domains, with $\overline{\os}\cup\overline{\of}=\overline{\om}$ and $\of\cap\os=\emptyset$. 
The center of mass of the solid is denoted by $\xg$, its mass and inertia tensor by $M$ and $J(t)$. Without loss of generality we assume that $M=1$. Then
$$\xg=\int_\os \rs x\, dx,\quad J(t)=\int_\os \rs(r^2\I-r\otimes r)\, dx$$
where $r(x,t)=x-\xg$. The system is subject to a body density force $g$ (usually gravity). 
\subsection{Weak formulation} The basic formulation of this fluid-solid coupling is the following~: given initial conditions,
\begin{equation}
x'_G(0)=v_g^0,\quad\omega_u(0)=\omega_u^0,\quad u = u^0,\quad\Omega_s(0)=\Omega_s^0\label{ci}
\end{equation}
supplemented with
\begin{equation}
x_G(0)=\int_{\Omega_s^0} x\, dx, \qquad X_s(x,0)=x,
\label{ciaux}
\end{equation}
find $t\to \os$ and $(x,t)\to(u(x,t), p(x,t))$ solution for $t>0$ of
\begin{align}
\rf(u_t+(u\cdot\nabla) u)-2\mu\dv (D(u))+\nabla p=\rf g&\qquad\text{on }\of,
\label{NS1}\\
\div u =0&\qquad\text{on }\of,
\label{NS2}\\
u = 0&\qquad\text{on }\dom,
\label{BCNS1}\\
u = x_G'+\omega_u\times r&\qquad\text{on }\dos,
\label{BCNS2}\\
x_G''(t)=g+\ds\int_{\dos}(\Sigma \, n) \, ds,
&\label{ConsMom1}\\
J(t)\omega_u'(t)=-\ou\times(J(t)\,\ou)+\ds\int_{\os}\rs (r\times g) \, dx+\ds\int_{\dos}r\times (\Sigma\, n) \,ds,
\label{ConsMom2}\\
\os=X_s(t,\Omega_s^0),
\label{bougeos}\\
\frac{\partial X_s}{\partial t}=x_G'(t)+\ou\times r(X_s(t),t),
\label{bougeX}
\end{align}
where $n$ denotes the unit outward normal on $\partial\Omega_s(t)$, and $\Sigma$ is the fluid stress tensor.\\
In this formulation  the last two equations  describe the rigid motion of $\os$. In order to give a weak formulation of this problem, let us introduce some function spaces. From now on, $u$ will denote the velocity field on the whole computational domain $\om$. We define
$$\mathcal{V}=\{u\in H^1_0(\Omega),\div u=0\},\quad \mathcal{H}=\{u\in L^2(\Omega),\;\div u=0,\;u\cdot n=0\text{ on }\dom\}$$
and, with the notations of \cite{SanMarStaTuc02} extended to the three dimensional  case,
\begin{multline*} \mathcal{K}(t)=\{u\in \mathcal{V},\;D(u)=0 \text{ in }\om_s(t)\}\\=\{u\in  \mathcal{V},\;\exists (V_u,\omega_u)\in\R^3\times \R,\; u=V_u+\omega_u\times r \text{ in }\om_s(t)\}.\end{multline*}
Next we define a density on the whole domain by setting
$\rho=\rho_s\chi_{\Omega_s(t)}+\rho_f\chi_{\Omega_f(t)}$, where $\chi_A$ denotes the characteristic function of set $A$, which takes value $1$ inside $A$ and $0$ outside. Let us note $Q=\Omega\times]0,T[$. Then the weak formulation is the following \cite{Hillairet2006}: given initial conditions $H^0=\chi_{\Omega_s^0}$, $\rho^0=\rho_sH^0+\rho_f(1-H^0)$ and $u=u^0\in \mathcal{K}(0)$,  find  $(x,t)\to(\rho(x,t),u(x,t),H(x,t))$ such that
\begin{equation}\label{forme_faible}
\begin{cases}
u\in L^{\infty}(0,T, \mathcal{H}) \cap L^2(0,T, \mathcal{V}),\quad H,\rho\in\mathcal{C}(0,T;L^q(\Omega))\;\forall q\ge 1,\\
u(t)\in \mathcal{K}(t) \text{ for a.e. }t\in]0,T[,\text{ with }\os=\{x\in\Omega,\;H(x,t)=1\},\\
\forall\xi\in H^1(Q)\cap L^2(0,T;\mathcal{K}(t)),\\
\qquad\displaystyle\int_{\Omega}\left[\rho u\cdot\partial_t \xi+(\rho (u\cdot\nabla)u-2\mu D(u)):D(\xi)+\rho g\cdot\xi\right]\, dx=\dfrac{d}{dt}\int_{\Omega}\rho u\cdot\xi \, dx,\\
\forall\psi\in\mathcal{C}^1(Q),\;\psi(T)=0,\\
\qquad\displaystyle\int_0^T\int_{\Omega}H\dfrac{\partial \psi}{\partial t} +H u\cdot\nabla\psi \, dx dt+ \displaystyle\int_{\Omega} H^0 \psi(0) \, dx=0,\\
\qquad\displaystyle\int_0^T\int_{\Omega}\rho\dfrac{\partial \psi}{\partial t} +\rho u\cdot\nabla\psi \, dx dt+ \displaystyle\int_{\Omega} \rho^0 \psi(0) \, dx=0.
\end{cases}
\end{equation}
Note that we could equivalently have defined $\rho=\rho_sH+\rho_f(1-H)$, as $\rho^0$ is piecewise constant, and transported by the same velocity field than $H$. 
\subsection{Penalized problem}
For $\eta>0$, we consider the following penalized problem:\\
given $(\re(0)=\rho^0, \ue(0)=\ue^0, \He(0)=\chi_{\Omega_s^0})$, to find $(\re,\ue,\pe,\He)$, with
$$\rho_{\eta},H_\eta\in L^\infty(]0,T[\times\Omega),\quad u_\eta\in L^\infty(0,T;\mathcal{H})\cap L^2(0,T;\mathcal{V}),\quad p_\eta\in L^2(Q)$$
solution on $Q$ of
\begin{eqnarray}
\rho_{\eta}\left(\dfrac{\partial u_{\eta}}{\partial t} +(u_{\eta}\cdot\nabla)u_{\eta}\right)-2\mu \dv(D(u_{\eta}))+\nabla p_{\eta}+\dfrac{1}{\eta}\re H_{\eta}(u_{\eta}-u_{\eta ,s})=\rho_{\eta}g \label{ns1penal}\\
\dv u_{\eta}=0\label{ns2penal}\\
u_{\eta ,s}=\frac1\Me\displaystyle\int_\Omega \rho_{\eta} u_{\eta} H_{\eta} \, dx
+\left(J_{\eta}^{-1}\,\displaystyle\int_\Omega \rho_{\eta} (r_{\eta}\times u_{\eta}) H_{\eta} \, dx\right)\times r_{\eta}\label{uspenal}\\
{\rho_{\eta}}_t+u_{\eta}.\nabla \rho_{\eta}=0\label{rhopenal}\\
{H_{\eta}}_t+u_{\eta ,s}.\nabla H_{\eta}=0\label{hpenal}
\end{eqnarray}
We set $\oes=\{x\in\Omega,\;H_\eta(x,t)=1\}$.
In equation (\ref{uspenal}) we divided the first term by $M_\eta= \int_\om\re\He \, dx$, which is not constant in time in general. On contrary we have $|\oes|=\int_\om \He \, dx = |\Omega_s^0|$ since $\ues$ is divergence free and $\He$ vanishes on $\dom$ (we assumed no contact of the solid with $\dom$). The inertia tensor is defined as
$$\Je=\int_\om\re\He(r_{\eta}^2\I-r_{\eta}\otimes r_{\eta})\, dx=\int_\oes\re(r_{\eta}^2\I-r_{\eta}\otimes r_{\eta})\, dx.$$
with $r_{\eta}=x-x_{G\eta}=x-\int_\om\re\He x\, dx$.\\
For $a\in\R^3\setminus\{0\}$, $a^T\Je a=\int_\oes\re|r_{\eta}\times a|^2\, dx\ge\min(\rs,\rf)\int_\oes|r_{\eta}\times a|^2\, dx$ (see estimate (\ref{bornerhoH})). This last quantity being strictly positive for an open nonempty integration set, $J_\eta$ is nonsingular (we recall that $|\oes|=|\Omega_s^0|>0$).

Before stating our convergence result, a few remarks are in order. 

First one may wonder about the well-posedness of the above problem. However it will directly result from the \textit{a priori} estimates and convergence arguments given in the following that this problem does have at least a weak solution. Indeed, these arguments could easily be used to show the convergence of the solutions to a linearized  version -  or finite-dimension approximation -  of (\ref{ns1penal})-(\ref{hpenal}). Next we can observe that in this model we penalize the difference between $\ue$ and the projection of $\ue$ onto velocity fields rigid in the solid domain, namely $\ues$ (see lemma \ref{propR1} below). The density is transported with the original velocity field so that estimates on the Navier-Stokes equations are easier to obtain. The characteristic function is transported by the rigid velocity so that the shape of $\oes$ remains undeformed (this is exactly the Eulerian counterpart of  (\ref{bougeos}-\ref{bougeX})). As observed in \cite{Coquerelle2008} this has a practical importance (in particular it means that the rigid solid can be recovered exactly through simple algebra from its initial shape). As far as numerical analysis is concerned, it also provides "for free"  regularity properties on the computed rigid body, as soon as the initial body is smooth. The price to pay is that the level sets of $\re$ and $\He$ do not coincide, i.e. in  general we do not have $\re\He=\rs$ as in the non penalized formulation.  Note also that in principle we should prescribe a boundary value for $\He$ on $\dom$ when $\ues$ is inward. Since our analysis is restricted to times when the solid body does not approach the boundary of the computational box, we can take this boundary value to be zero, which amounts to solve (\ref{hpenal}) on $\R^n$ and take its restriction to $\om$.

In the following sections we will prove the convergence of at least a subsequence of $(\re,\ue,\pe,\He)$ to the weak solution defined above. Next section starts with some \textit{a priori} estimates which will provide weak convergence of subsequences. In section \ref{sec4} we will have to use more sophisticated tools adapted from \cite{SanMarStaTuc02} to get some strong convergence in $\ue$ which will allow us to pass to the limit in nonlinear terms of $(P_\eta)$. More precisely we prove the following result.

\begin{theorem}\label{cvresult}
Under the regularity assumptions of section \ref{sec2}, let $(\re,\ue,\pe,\He)$ a solution of $(P_\eta)$. Then there exists a subsequence of $(\re,\ue,\He)$ and functions $(\rho,u,H)$  such that 
$$\re\to\rho,\quad\He\to H \text{ strongly in }C(0,T;L^q(\Omega))\text{ for all }q\ge 1,$$
$$\ue\to u \text{ strongly in }L^2(Q)\text{ and weakly in }L^2(0,T;H_0^1(\Omega))\cap L^\infty(0,T;L^2(\Omega))$$
and such that $(\rho,u,H)$, is a solution of (\ref{forme_faible}).
\end{theorem}
Before proceeding to the proof, let us point out a few remarks.
For a sake of simplicity in the notations we have stated our penalization method and theorem for a single rigid body. It will be apparent from the proof below that it readily extends to the case of several bodies.
Furthermore, the time to which the convergence result is restricted, is essentially the time for which contact of the rigid body do not touch the boundary of $\Omega$ (in the case of several bodies it would be the time on which we can ensure that contact between bodies do not happen). As a result if we consider periodic boundary conditions and a single body convergence holds for all times.

\section{Proof of theorem \ref{cvresult}}
The following lemma states that $\ues$, as defined in $(P_\eta)$, is the projection of $\ue$ onto velocity fields which are rigid on $\oes$. 
\begin{lemma}\label{propR1}
Let $\xi$ be a rigid velocity field, i.e. such that $\xi(x)=V_\xi+\omega_\xi\times r(x)$ for some constant vectors $V_\xi\in\R^3$ and $\omega_\xi\in\R^3$. Then if $\ues$ is defined by (\ref{uspenal}) there holds
\begin{equation}
\displaystyle\int_{\Omega}\rho_{\eta}H_{\eta}(u_{\eta}-u_{\eta,s})\cdot\xi \, dx = 0.
\end{equation}
Moreover, the result holds if $\xi$ is a time dependent velocity field rigid in $\oes$ at time $t$.
\end{lemma}
\begin{proof}
Let the mean translation and angular velocities be defined as
\begin{equation*}
V_u=\frac1\Me\displaystyle\int_\Omega \rho_{\eta}H_{\eta} u_{\eta} \, \, dx
\qquad\omega_u=J^{-1}_{\eta}\,\displaystyle\int_\Omega \rho_{\eta}H_{\eta} (r_{\eta}\times u_{\eta}) \, dx
\end{equation*}
then
\begin{eqnarray*}
\displaystyle\int_{\Omega}\rho_{\eta}H_{\eta}(u_{\eta}-u_{\eta,s})\cdot\xi \, dx & = & \displaystyle\int_{\Omega}\rho_{\eta}H_{\eta}\left[u_{\eta}-\left(V_{u}
+\omega_{u}\times r_{\eta}\right)\right]\cdot\left[V_{\xi}+\omega_{\xi}\times r_{\eta}\right] \, dx\notag\\
& = & V_{\xi}\cdot\displaystyle\int_{\Omega}\rho_{\eta}H_{\eta}u_{\eta} \, dx
+\omega_{\xi}\cdot\displaystyle\int_{\Omega}\rho_{\eta}H_{\eta}(r_{\eta} \times u_{\eta}) \, dx
- V_{u}\cdot V_{\xi} \displaystyle\int_{\Omega}\rho_{\eta}H_{\eta} \, dx\\
& - & V_{u} \cdot\left(\omega_{\xi}\times\displaystyle\int_{\Omega}\rho_{\eta}H_{\eta}r_{\eta} \, dx\right)
- V_{\xi} \cdot\left(\omega_{u}\times\displaystyle\int_{\Omega}\rho_{\eta}H_{\eta}r_{\eta} \, dx\right)\notag\\
& - & \displaystyle\int_{\Omega}\rho_{\eta}H_{\eta}(\omega_{u} \times r_{\eta})\cdot(\omega_{\xi} \times r_{\eta}) \, dx\notag\\
& = & V_{\xi}\cdot(M_{\eta} V_{u})
+\omega_{\xi}\cdot(J_{\eta}\, \omega_{u})
- V_{u} \cdot(M_{\eta}V_{\xi})\\
& - & V_{u}\cdot \left(\omega_{\xi}\times\displaystyle\int_{\Omega}\rho_{\eta}H_{\eta}r_{\eta} \, dx\right)
-  V_{\xi}\cdot \left(\omega_{u}\times\displaystyle\int_{\Omega}\rho_{\eta}H_{\eta}r_{\eta} \, dx\right)\notag\\
& - & \displaystyle\int_{\Omega}\rho_{\eta}H_{\eta}(\omega_{u} \times r_{\eta})\cdot(\omega_{\xi} \times r_{\eta}) \, dx.
\end{eqnarray*}
As $(\omega_{u} \times r_{\eta})\cdot(\omega_{\xi} \times r_{\eta})=(\omega_{\xi}\cdot\omega_{u})r_{\eta}^2-(r_{\eta}\cdot\omega_{\xi})(r_{\eta}\cdot\omega_{u})$, we have $\displaystyle\int_{\Omega}\rho_{\eta}H_{\eta}(\omega_{u} \times r_{\eta})\cdot(\omega_{\xi} \times r_{\eta}) \, dx=\omega_{\xi}\cdot(J_{\eta}\, \omega_{u})$.\\
Finally, by definition of $r_{\eta}$, $\displaystyle\int_{\Omega}\rho_{\eta}H_{\eta}r_{\eta} \, dx=0$, and we get
\begin{eqnarray*}
\displaystyle\int_{\Omega}\rho_{\eta}H_{\eta}(u_{\eta}-u_{\eta,s})\cdot\xi \, dx & = &  \omega_{\xi}\cdot(J_{\eta}\,\omega_{u})-\omega_{\xi}\cdot(J_{\eta}\,\omega_{u})=0.
\end{eqnarray*}
\end{proof}
\subsection{Estimates for transport and Navier-Stokes equations}
In all the sequel, $C$ denotes a positive constant. At this stage, we consider a given time interval $[0,T]$. The value to which $T$  must be restricted will be given later in this section.

Standard estimates for transport equations (\ref{rhopenal}) and (\ref{hpenal}) show that $\re$ and $\He$ are bounded in $L^\infty(0,T;L^\infty(\om))$. More precisely, for all time $t\in[0,T]$, 
\be\rho_{min}:=\min(\rs,\rf)\le \re(x,t)\le\max(\rs,\rf)\qquad \He(x,t)\in\{0,1\}\quad\text{a.e. }x\in\om.\label{bornerhoH}\ee
Thus, up to extracting a subsequence, we can assume that
\begin{equation}\label{rho_cv_1}
\rho_{\eta} \rightharpoonup \rho \text{ in } L^{\infty}(0,T,L^{\infty}(\Omega)) \text{ weak*,}
\end{equation}
and
\begin{equation}\label{h_cv_1}
H_{\eta} \rightharpoonup H \text{ in } L^{\infty}(0,T,L^{\infty}(\Omega)) \text{ weak*,}
\end{equation}
where $H$ and $\rho$ satisfy the bounds  (\ref{bornerhoH}). We set $\os=\{x\in\Omega,\;H(x,t)=1\}$. Concerning Navier-Stokes equations, 
multiplying  (\ref{ns1penal}) by $u_{\eta}$ and integrating on $\Omega$, we get:
\begin{eqnarray*}
& & \displaystyle\int_{\Omega} \rho_{\eta}\left(\dfrac{\partial u_{\eta}}{\partial t} +(u_{\eta} \cdot\nabla) u_{\eta}\right) \cdot u_{\eta} \, dx 
- 2\mu \displaystyle\int_{\Omega} \dv(D(u_{\eta})) \cdot u_{\eta} \, dx 
+ \displaystyle\int_{\Omega} u_{\eta}\cdot \nabla p_{\eta} \, dx \\
& + & \dfrac{1}{\eta} \displaystyle\int_{\Omega} \re H_{\eta}(u_{\eta}-u_{\eta ,s})\cdot u_{\eta} \, dx
= \displaystyle\int_{\Omega} \rho_{\eta}g\cdot u_{\eta} \, dx.
\end{eqnarray*}
Classically we have from incompressibility and homogeneous boundary conditions on $\ue$, 
\be
\displaystyle\int_{\Omega} \rho_{\eta}\left(\dfrac{\partial u_{\eta}}{\partial t} +(u_{\eta} \cdot\nabla) u_{\eta}\right) \cdot u_{\eta} \, dx = \frac12\int_{\Omega} \dfrac{\partial (\rho_{\eta}|u_{\eta}|^2)}{\partial t}  \, dx.
\ee
From Lemma \ref{propR1} we get
\begin{equation*}
\displaystyle\int_{\Omega} \re H_{\eta}(u_{\eta}-u_{\eta ,s})\cdot u_{\eta} \, dx=\displaystyle\int_{\Omega} \re H_{\eta}(u_{\eta}-u_{\eta ,s})^2 \, dx,
\end{equation*}
and since $\ue$ is divergence free and vanishes on $\dom$,
\begin{equation*}
\displaystyle\int_{\Omega}  u_{\eta}\cdot\nabla p_{\eta} \, dx=0.
\end{equation*}
Collecting terms we get, since from (\ref{bornerhoH}) $\sqrt{\He}=\He$,
\begin{equation*}
\dfrac{1}{2}\dfrac{d}{dt} {\Vert \sqrt{\rho_{\eta}}u_{\eta} \Vert}^2_{L^2(\Omega)}
+\mu {\Vert D(u_{\eta}) \Vert}^2_{L^2(\Omega)}
+\dfrac{1}{\eta} {\Vert \sqrt{\re}H_{\eta}(u_{\eta}-u_{\eta ,s}) \Vert}^2_{L^2(\Omega)} \leq \Vert \sqrt{\rho_{\eta}}u_{\eta} \Vert_{L^2(\Omega)}{\Vert g \Vert_{L^{\infty}(Q)}}{\Vert \rho_{\eta} \Vert^{\frac{1}{2}}_{L^2(Q)}}
\end{equation*}
which upon time integration on $[0,T]$ gives
\begin{eqnarray*}
{\Vert \sqrt{\rho_{\eta}}(t)u_{\eta}(t) \Vert}^2_{L^2(\Omega)}
& + & 2\mu {\Vert D(u_{\eta}) \Vert}^2_{L^2(Q)}
+\dfrac{2}{\eta} {\Vert \sqrt{\re}H_{\eta}(u_{\eta}-u_{\eta ,s}) \Vert}^2_{L^2(Q)}\\
& \leq & {\Vert \sqrt{\rho_{\eta 0}}u_{\eta 0} \Vert}^2_{L^2(\Omega)}+C \displaystyle\int_0^T \Vert \sqrt{\rho_{\eta}}(s)u_{\eta}(s) \Vert_{L^2(\Omega)}ds.
\end{eqnarray*}
Applying Gronwall Lemma, Poincar\'e inequality and bounds from (\ref{bornerhoH}) gives the following estimates~:
\begin{equation}\label{u_bornee_1}
u_{\eta} \text{ bounded in } L^2(0,T,H^1_0(\Omega)),
\end{equation}
\begin{equation}\label{u_bornee_2}
\sqrt{\rho_{\eta}}u_{\eta}\text{ and }u_{\eta}\text{ bounded in } L^{\infty}(0,T,L^2(\Omega)),
\end{equation}
\begin{equation}\label{hu_bornee_1}
\dfrac{1}{\sqrt{\eta}}\sqrt\re H_{\eta}(u_{\eta}-u_{\eta ,s})\text{ and } \dfrac{1}{\sqrt{\eta}} H_{\eta}(u_{\eta}-u_{\eta ,s})\text{ bounded in } L^2(0,T,L^2(\Omega)).
\end{equation}
Thus we can extract subsequences from $\rho_{\eta}$, $u_{\eta}$ and $H_{\eta}$, still denoted by $\rho_{\eta}$, $u_{\eta}$ and $H_{\eta}$, such that
\begin{equation}\label{u_cv_1}
u_{\eta} \rightharpoonup u \text{ in } L^2(0,T,H^1_0(\Omega)) \text{ weak,}
\end{equation}
\begin{equation}\label{u_cv_2}
\sqrt{\rho_{\eta}}u_{\eta} \rightharpoonup \chi \text{ and }u_{\eta} \rightharpoonup u\text{ in } L^{\infty}(0,T,L^2(\Omega)) \text{ weak*,}
\end{equation}
\begin{equation}\label{u_cv_3}
\sqrt{\re}H_{\eta}u_{\eta} -\sqrt\re H_{\eta}u_{\eta,s} \to 0\text{ and } H_{\eta}u_{\eta} - H_{\eta}u_{\eta,s} \to 0\text{ in } L^2(0,T,L^2(\Omega)) \text{ strong. }
\end{equation}
The identification of $\chi$ with $\sqrt{\rho} u$ results from strong convergence results proved by Lions and DiPerna on transport equations. \cite{DiPerna1989} theorem II.4, (\ref{u_cv_1}) and incompressibility imply
\begin{equation}\label{rho_cv_2}
\rho_{\eta} \to \rho \text{ in } C(0,T,L^q(\Omega)) \text{ strong }\forall q\in [1,+\infty[\\\\
\end{equation}
with $\rho$ solution of
\begin{equation*}
\begin{cases}
{\rho}_t+u\cdot\nabla \rho=0 &\text{on }\Omega\times]0,T[,\\
\rho=\rho_{0} &\text{on } \Omega\times\{0\}.
\end{cases}
\end{equation*}
From this strong convergence we can pass to the limit in the product $\sqrt{\re}\ue$~: given $v\in L^q(0,T;L^r(\om))$ with $q>2$ and $r>\frac65$, we write
$$\int_0^T\int_\om(\sqrt{\rho_{\eta}}u_{\eta} - \sqrt\rho u)v \, dx dt=\int_0^T\int_\om (u_{\eta} -  u)\sqrt\rho v \, dx dt+\int_0^T\int_\om (\sqrt\re - \sqrt \rho) \ue v \, dx dt.$$
From the injection of $H^1$ into $L^6$ in dimension less or equal to $3$ the first integral converges toward $0$. For the second integral we use the strong convergence (following (\ref{rho_cv_2})) of $\sqrt{\re}$ in $L^s$ for a $s$ such that $\ue v$ is in $L^{s'}$ where $s'$ is the conjugate exponent of $s$. Thus we have
\begin{equation}
\sqrt{\rho_{\eta}}u_{\eta} \rightharpoonup \sqrt\rho u \text{ in } L^{q}(0,T,L^r(\Omega)) \text{ weak, for all }q<2,r<6.
\end{equation}

\subsection{Setting T and passing to the limit in the rigid velocity}
This rigid velocity is defined by
\begin{equation*}
u_{\eta ,s}(x,t)=u_{\eta,G}(t)+\omega_{\eta}(t)\times r_{\eta}(x,t),
\end{equation*}
with
\begin{equation*}
u_{\eta,G}(t)=\frac{1}{M_\eta}\int_\Omega \rho_{\eta} u_{\eta}H_{\eta} \, dx\quad\text{ and }
\quad\omega_{\eta}(t)=J_\eta^{-1}\,\displaystyle\int_\Omega \rho_{\eta} (r_{\eta}\times u_{\eta}) H_{\eta} \, dx.
\end{equation*}
First we note that $M_\eta=\int_\Omega \re\He \, dx$ is bounded from below independently of $\eta$ since $\re$ is bounded from below and $\int_\Omega \He \, dx=|\Omega_s^0|>0$ does not depend on $\eta$ or $t$. From the bounds on $\re$, $\He$ and $\ue$ it is straightforward to show that \begin{equation*}
u_{\eta ,G}(t)\text{ bounded in }L^{\infty}(0,T).
\end{equation*}
Likewise, from the definition of $\Je$ we observe that for $a\in\R^3\setminus\{0\}$ 
$$a^T\Je a\ge\min(\rs,\rf)\int_\os|r_{\eta}\times a|^2\, dx>0.$$
Moreover, the initial solid is regular and transported by a rigid velocity. We thus know that  there is a ball of radius $R>0$ centered on the center of gravity $x_{G\eta}$ included into $\oes$. Then the above estimates implies
$$a^T\Je a\ge\min(\rs,\rf)\int_{B(x_{G\eta},R)}|r_{\eta}\times a|^2\, dx=\min(\rs,\rf)\int_{B(0,R)}|x\times a|^2\, dx=C(R)|a|^2$$
with $C(R)=\frac{2R^5\pi}{15}>0$. Taking $a=\Je^{-\frac12}b$ ($\Je$ is symmetric) we get for all $b\in  \R^3\setminus\{0\}$, 
$$b^T\Je^{-1} b=|\Je^{-\frac12}b|^2\le\frac1{C(R)}|b|^2,$$
which proves that each coefficient of $\Je^{-1}$ is bounded independently of $\eta$ and $t$. From the bounds on $\ue$, $\He$ and $\re$ this implies that
\begin{equation*}
\omega_{\eta}(t)\text{ is bounded in }L^{\infty}(0,T).
\end{equation*}
In particular this implies that the solid velocity $u_{\eta,s}$ is bounded in $L^{\infty}$ by some constant $M$ independent of $\eta$ and time. We can now define the maximum time for which the convergence result will be proved. If we denote by $d_0$ the initial distance between  solid and the boundary $\partial\Omega$ then choosing for instance $T=d_0/2M$ ensures that the body will not touch the boundary for $t\in [0,T]$. In all the sequel we will assume this value of $T$.

From the above estimates we can ensure that    there exists $u_{G}(t)$ and $ \omega(t)$ in $L^\infty(0,T)$ such that, up to the extraction of subsequences,  
\begin{equation*}
u_{\eta,s} \rightharpoonup u_s:=u_G+\omega\times r \text{ in } L^{\infty}(0,T,L^\infty(\om)) \text{ weak*.}
\end{equation*}
Now we point out that taking the gradient of the rigid velocity field $\ues$ gives
$$\nabla\ues = \begin{pmatrix}0&-\omega_\eta^3&\omega_\eta^2\\\omega_\eta^3&0&-\omega_\eta^1\\-\omega_\eta^2&\omega_\eta^1&0\end{pmatrix}$$
so that the $\nabla\ues$ (and all subsequent space derivatives) is also bounded in $L^\infty(0,T;L^\infty(\om))$. In particular
\begin{equation}\label{us_cv_1}
u_{\eta,s} \rightharpoonup u_s \text{ in } L^{2}(0,T,W^{1,\infty}(\om)) \text{ weak*.}
\end{equation}

We now wish to prove that $u_s$ (or equivalently $u_G$ and $\omega$) has a similar structure as $\ues$, that is, to pass to the limit in the expression defining $\ues$. Using (\ref{us_cv_1}), the already mentioned compactness results of \cite{DiPerna1989}, applied to the transport equation on $\He$ now gives\begin{equation}\label{h_cv_2}
H_{\eta} \to H \text{ a.e. and in } C(0,T,L^p(\Omega)) \text{ strong }\forall p\in [1,+\infty[\\\\
\end{equation}
with $H$ verifying
\begin{equation*}
\begin{cases}
{H}_t+u_{s}.\nabla H=0 \text{         on } \R^n\times(0,T)\\
H=H_{0} \text{         on } \R^n\times\{0\}.
\end{cases}
\end{equation*}
Note that this Cauchy problem has been set in $\R^n$ because $u_s$ does not vanish on $\partial\Omega$, but $H$ vanishes outside $\Omega$. However we can prove by passing to the limit in (\ref{u_cv_3}) that $Hu=Hu_s$. Thus $\mbox{div$(uH)$}=\mbox{div$(u_sH)$}$ and $H$ verifies a transport equation with velocity field $u$ on $\Omega$ (note that no boundary conditions are needed on $\partial \Omega$ since $u$ is zero on the boundary). 

This convergence gives us the strong convergence of $r_{\eta}$ in $\mathcal{C}(0,T,L^p(\Omega))$, $\forall p \geq 1$, and from (\ref{rho_cv_2}),(\ref{h_cv_2}) and (\ref{u_cv_1}) we can easily pass to the limit in the expression of $u_{\eta,G}$ and $\omega_\eta$ to get 
\begin{equation*}
u_{G}(t)=\dfrac{\displaystyle\int_\Omega \rho u H \, dx}{\displaystyle\int_\Omega \rho H \, dx}\quad\text{ and }
\quad\omega(t)={\left(\displaystyle\int_\Omega \rho (r^2\mathbb{I}-r\otimes r) H \, dx\right)}^{-1}\displaystyle\int_\Omega \rho (r\times u) H \, dx.
\end{equation*}

\subsection{Strong convergence of $u_{\eta}$}
\label{sec4}
The remaining part of the proof is more technical since we aim to prove the strong convergence of at least a subsequence of $\ue$ in order to be able to pass to the limit in the inertial term of Navier-Stokes equations. Classically, this is obtained thanks to a Fourier transform in time which provides an estimate on some fractional time derivative of $\ue$ which brings compactness \cite{Lions, Temam1979}. Here these technics can not be used since the solid is moving. We instead rely on tools developed in \cite{SanMarStaTuc02}.

Thereafter we will use, for $\sigma>0$ and $r\in[0,1]$, the following notations 
\begin{itemize}
\item $\Omega_{s,\sigma}(t)=\{x\in\Omega, d(x,\Omega_s(t))<\sigma\}$,
\item $\mathcal{V}^0=\{v\in L^2(\Omega),\;\dv v=0,\;v\cdot n=0 \text{ on } \partial\Omega\}$,
\item $\mathcal{V}^r=\{v\in H^r(\Omega),\;\dv v=0,\;v=0 \text{ on } \partial\Omega\}$ (for $r>0$),
\item $\mathcal{K}^r_{\sigma}(t)=\{v(t)\in \mathcal{V}^r,D(v(t))=0 \text{ in } \mathcal{D}'(\Omega_{s,\sigma}(t))\}$ (which is a closed subset of $\mathcal{V}^r$),
\item $P^r_{\sigma}(t)$ the orthogonal projection of $\mathcal{V}^r$ on $\mathcal{K}^r_{\sigma}(t)$.
\end{itemize}
To prove the strong convergence of a subsequence of $\ue$ in  $L^2(Q)$ we write
\begin{equation*}
\displaystyle\int_0^T\int_{\Omega} \vert u_{\eta}-u \vert^2 \, dx dt  \leq  \dfrac{1}{\rho_{min}}\left(\displaystyle\int_0^T\int_{\Omega} \vert \rho(u^2_{\eta}-u^2) \vert \, dx dt
+\displaystyle\int_0^T\int_{\Omega} \vert 2\rho u \cdot (u-u_{\eta}) \vert \, dx dt\right).
\end{equation*}
From (\ref{u_cv_2}) the second integral on the right side converges to 0, thus
\begin{equation*}
\displaystyle\int_0^T\int_{\Omega} \vert u_{\eta}-u \vert^2 \, dx dt   \leq  \dfrac{1}{\rho_{min}}\left(\displaystyle\int_0^T\int_{\Omega} \vert \rho_{\eta}u^2_{\eta}-\rho u^2 \vert \, dx dt
+\displaystyle\int_0^T\int_{\Omega} \vert (\rho_{\eta}-\rho)u^2_{\eta} \vert \, dx dt\right).
\end{equation*}
Moreover, by (\ref{u_bornee_1}) and (\ref{rho_cv_2}) the second integral on the right hand side converges to $0$, and
\begin{eqnarray*}
\displaystyle\int_0^T\int_{\Omega} \vert u_{\eta}-u \vert^2 \, dx dt & \leq & \dfrac{1}{\rho_{min}}\left(\displaystyle\int_0^T\int_{\Omega} \vert \rho_{\eta}u_{\eta}\cdot  P^r_{\sigma}(u_{\eta})-\rho u\cdot  P^r_{\sigma}(u) \vert \, dx \, dt\right.\notag\\
& + & \displaystyle\int_0^T\int_{\Omega} \vert \rho_{\eta}u_{\eta}\cdot (u_{\eta}-P^r_{\sigma}(u_{\eta})) \vert \, dx \,dt
+  \left.\displaystyle\int_0^T\int_{\Omega} \vert \rho u\cdot (P^r_{\sigma}(u)-u) \vert \, dx \,dt + l_{\eta}\right) \notag\\
& \leq & \dfrac{1}{\rho_{min}}(\Vert \rho_{\eta}u_{\eta}\cdot P^r_{\sigma}(u_{\eta})-\rho u\cdot P^r_{\sigma}(u) \Vert_{L^1(Q)}\notag\\
& + & \Vert \rho_{\eta} \Vert_{L^{\infty}(Q)}\Vert u_{\eta} \Vert_{L^2(Q)}\Vert P^r_{\sigma}(u_{\eta})-u_{\eta} \Vert_{L^2(Q)}\notag\\
& + & \Vert \rho  \Vert_{L^{\infty}(Q)}\Vert u \Vert_{L^2(Q)}\Vert P^r_{\sigma}(u)-u \Vert_{L^2(Q)}+ l_{\eta}
\end{eqnarray*}
where $l_{\eta}\rightarrow 0$ when $\eta\rightarrow 0$.
Finally, as $\rho_{\eta}$ is bounded in $L^{\infty}(0,T,L^{\infty}(\Omega))$ and using (\ref{u_bornee_1}), we get
\begin{eqnarray}\label{decomp}
\displaystyle\int_0^T\int_{\Omega} \vert u_{\eta}-u \vert^2 \, dx dt & \leq & \dfrac{1}{\rho_{min}}(\Vert \rho_{\eta}u_{\eta}\cdot P^r_{\sigma}(u_{\eta})-\rho u\cdot P^r_{\sigma}(u) \Vert_{L^1(Q)}\notag\\
& + & C\Vert P^r_{\sigma}(u_{\eta})-u_{\eta} \Vert_{L^2(Q)}\notag\\
& + & C\Vert P^r_{\sigma}(u)-u \Vert_{L^2(Q)} + l_{\eta}
\end{eqnarray}
This decomposition shows that the sought convergence essentially amounts to prove  that (up to the extraction of subsequences)
\begin{align}
\lim_{\sigma\to 0}\Vert P^r_{\sigma}(u)-u \Vert_{L^2(Q)}&=0\label{part1}\\
\lim_{\sigma\to 0}\lim_{\eta\to 0}\Vert P^r_{\sigma}(u_{\eta})-u_{\eta} \Vert_{L^2(Q)}&=0\label{part2}\\
\lim_{\sigma\to 0}\lim_{\eta\to 0}\Vert \rho_{\eta}u_{\eta}\cdot P^r_{\sigma}(u_{\eta})-\rho u\cdot P^r_{\sigma}(u) \Vert_{L^1(Q)}&=0\label{part3}
\end{align}
To prove (\ref{part1}-\ref{part3}) we will make use of some Lemma that we state now.
\begin{lemma}\label{egorov}
Let $(f_n)$ be a sequence of functions bounded in $L^p(0,T)$ for some $p>2$ and converging to $0$ almost everywhere on $[0,T]$. Then $(f_n)$ converges strongly to $0$ in $L^2(0,T)$.
\end{lemma}
\begin{proof}
Let $\varepsilon >0$. From Egorov theorem (\cite{Brezis1992}, p.75), the almost everywhere convergence implies that there exists $A_{\varepsilon}\subset [0,T]$ such that
\begin{equation*}
\begin{cases}
|[0,T]\backslash A_{\varepsilon}|<\varepsilon\\
f_n \to 0 \text{ uniformly on }A_{\varepsilon}
\end{cases}
\end{equation*}
which means
\begin{equation*}
\exists N\in\N, \forall n\geq N, \forall t\in A_{\varepsilon}, |f_n(t)|^2 < \varepsilon.
\end{equation*}
Therefore
\begin{equation*}
\displaystyle\int_{A_{\varepsilon}} (f_n(t))^2 dt \leq \varepsilon |A_{\varepsilon}| \leq \varepsilon T.
\end{equation*}
As from the bound in $L^p$ norm,
\begin{equation*}
\displaystyle\int_{[0,T]\backslash A_{\varepsilon}} (f_n(t))^2 dt \leq \displaystyle\int_{[0,T]\backslash A_{\varepsilon}} 1^q dt \displaystyle\int_{[0,T]} (f_n(t))^p dt \leq \varepsilon^q C
\end{equation*}
with $\frac{1}{q}+\frac{1}{p}=\frac{1}{2}$, we get
\begin{equation*}
\displaystyle\int_0^T (f_n(t))^2 dt \leq \varepsilon T + \varepsilon^q C
\end{equation*}
which proves the $L^2$ convergence.
\end{proof}
\begin{lemma}\label{stokes}
Let $u(t)\in H^1(\Omega_{s,\sigma}(t))$ such that $u_{|\partial\Omega_{s,\sigma}(t)}(t)=g(t)$ and $(w(t),p(t))\in H^1(\Omega\backslash\overline{\Omega_{s,\sigma}}(t))\times L^2(\Omega\backslash\overline{\Omega_{s,\sigma}}(t))$ solution of the Stokes problem 
\begin{equation*}
\begin{cases}
-\Delta w(t)+\nabla p(t)=0 &\text{on }\Omega\backslash\overline{\Omega_{s,\sigma}}(t),\\
\dv w(t)=0 &\text{on }\Omega\backslash\overline{\Omega_{s,\sigma}}(t),\\
w(t)=g(t) &\text{on }\partial\Omega_{s,\sigma}(t),\\
w(t)=0 &\text{on }\partial\Omega.
\end{cases}
\end{equation*}
Then there exists $\sigma_0>0$ and $C>0$, such that for all $\sigma<\sigma_0$, we have the following estimate:
\begin{equation*}
\qquad{\Vert w(t) \Vert}^2_{L^2(\Omega\backslash\overline{\Omega_{s,\sigma}}(t))} \leq C {\Vert u(t) \Vert}_{L^2(\Omega_{s,\sigma}(t))} {\Vert \nabla u(t) \Vert}_{L^2(\Omega_{s,\sigma}(t))}
\end{equation*}
\end{lemma}
The proof of this result is postponed to the appendix.
\begin{lemma}\label{u_limite_rigide}
The limits $H$, $u$ and $u_s$ defined in (\ref{h_cv_2}), (\ref{u_cv_1}) and (\ref{us_cv_1}) verify
\begin{equation}
Hu=Hu_s.
\end{equation}
\end{lemma}
\begin{proof}
First we claim that 
\begin{equation}\label{cv1}
H_{\eta}u_{\eta} \rightharpoonup Hu \text{ in } L^q(0,T,L^r(\Omega)) \text{ weak, with }q<2 \text{ and } r<6
\end{equation}
Indeed, let us introduce $v\in L^q(0,T,L^r(\Omega))$ with $q>2$ and $r>\frac{6}{5}$. We have
\begin{equation*}
\displaystyle\int_0^T\int_{\Omega}(H_{\eta}u_{\eta}-Hu)\cdot v\, dx dt=\displaystyle\int_0^T\int_{\Omega}H(u_{\eta}-u)\cdot v\, dx dt+\displaystyle\int_0^T\int_{\Omega}(H_{\eta}-H)u_{\eta}\cdot v\, dx dt
\end{equation*}
From the injection of $H^1$ into $L^6$ in dimension less or equal to 3 and (\ref{u_cv_1}) we get
\begin{equation*}
u_{\eta} \rightharpoonup u \text{ in } L^2(0,T,L^6(\Omega)) \text{ weak,}
\end{equation*}
and, since $H$ is bounded in $L^{\infty}(Q)$, 
\begin{equation*}
\displaystyle\lim_{\eta\to 0}\int_0^T\int_{\Omega}H(u_{\eta}-u)\cdot v\, dx \,dt =0.
\end{equation*}
Moreover from (\ref{h_cv_2}) we easily get
\begin{equation*}
\displaystyle\lim_{\eta\to 0}\int_0^T\int_{\Omega}(H_{\eta}-H)u_{\eta}\cdot v\, dx \,dt =0.
\end{equation*}
We next show that 
\begin{equation}\label{cv2}
H_{\eta}u_{\eta,s} \rightharpoonup Hu_s \text{ in } L^p(0,T,L^p(\Omega)) \text{ weak, }\forall p\in[1,+\infty[.
\end{equation}
Let $v\in L^p(0,T,L^p(\Omega))$ with $p>1$. We write
\begin{equation*}
\displaystyle\int_0^T\int_{\Omega}(H_{\eta}u_{\eta,s}-Hu_s)\cdot v\, dx dt=\displaystyle\int_0^T\int_{\Omega}H(u_{\eta,s}-u_s)\cdot v\, dx dt+\displaystyle\int_0^T\int_{\Omega}(H_{\eta}-H)u_{\eta,s}\cdot v\, dx \,dt.
\end{equation*}
As $H$ is bounded in $L^{\infty}(Q)$, with (\ref{us_cv_1}) we get:
\begin{equation*}
\displaystyle\lim_{\eta\to 0}\int_0^T\int_{\Omega}H(u_{\eta,s}-u_s)\cdot v\, dx \,dt =0
\end{equation*}
In addition, from (\ref{h_cv_2}) we  get
\begin{equation*}
\displaystyle\lim_{\eta\to 0}\int_0^T\int_{\Omega}(H_{\eta}-H)u_{\eta,s}\cdot v\, dx \,dt =0.
\end{equation*}
By (\ref{cv1}) and (\ref{cv2}) we thus deduce 
\begin{equation}\label{cv3}
H_{\eta}u_{\eta}-H_{\eta}u_{\eta,s} \rightharpoonup Hu-Hu_s \text{ in } L^q(0,T,L^r(\Omega)) \text{ weak, with }q<2 \text{ and } r<6
\end{equation}
Finally we recall that 
\begin{equation}\label{cv4}
H_{\eta}u_{\eta} - H_{\eta}u_{\eta,s} \to 0 \text{ in } L^2(Q) \text{ strong. }
and the desired result follows
\end{equation}
The result is finally obtained by identifying the limits in (\ref{cv3}) et (\ref{cv4}).\\
\end{proof}
\begin{lemma}\label{inclusion_omega_s}
\begin{equation}\label{distome}
\forall\sigma>0, \quad \exists\eta_0>0, \quad \forall\eta<\eta_0, \quad \forall t\in[0,T],\quad\oes\subset\Omega_{s,\sigma}(t)\textrm{ and }\Omega_{s}(t)\subset\Omega_{s,\sigma}^\eta(t).
\end{equation}
\end{lemma}
\begin{proof}
From (\ref{h_cv_2}) with $p=1$ we have
\begin{equation}
\forall\eps>0, \quad \exists\eta_0>0, \quad \forall\eta<\eta_0, \quad \forall t\in[0,T], \quad\int_\Omega|\He(x,t)-H(x,t)|\, dx<\eps,
\end{equation}
which means
\begin{equation}\label{mesinfeps}
\forall\eps>0, \quad \exists\eta_0>0, \quad \forall\eta<\eta_0, \quad \forall t\in[0,T], \quad|\oes\setminus\os|+|\os\setminus\oes|<\eps.
\end{equation}
By contradiction we suppose that we can find $\sigma_0>0$ such as $\forall\eta_0>0$, there exists $\eta<\eta_0$ and $t\in[0,T]$ for which at least one of the inclusions of (\ref{distome}) is false. This means that we can find  
$x_\eta(t)\in\Omega^{\eta}_s(t)$ such as $d(x_\eta(t),\Omega_s(t))>\sigma_0$. 
$\Omega^{\eta}_s(t)$ is a rigid deformation of $\Omega^{\eta}_s(0)$ so its boundary is
$\mathcal{C}^2$. Thus, there exists a sufficiently small $\rho$ independent of $\eta$, such as
for each point of $\Omega^{\eta}_s(t)$ there exists a ball 
of radius $\rho>0$ containing this point and included in $\Omega^{\eta}_s(t)$. Then there exists 
also a ball of radius $\bar{\rho}:=\min(\rho,\sigma_0/3)$ containing the point and included in
$\Omega^{\eta}_s(t)$. This latter ball is  included in 
$\Omega^{\eta}_s(t)\setminus\Omega_s(t)$. Indeed it contains a point at distance more than $\sigma_0$ from  $\Omega_s(t)$ and its diameter is less than $2\sigma_0/3$. 
We thus got that $$\exists\sigma_0,\;\forall\eta_0>0,\;\exists\eta>0,\;\exists t\in[0,T],\;
|\Omega^{\eta}_s(t)\setminus\Omega_s(t)|>\pi\bar{\rho}^2$$
with $\bar{\rho}$ independent of $\eta$ and $t$. This contradicts (\ref{mesinfeps}).
\end{proof}
\begin{lemma}\label{convergence_Pu_u}
\begin{equation}
\displaystyle\lim_{\sigma\to 0} \Vert P^r_{\sigma}(u)-u \Vert_{L^2(0,T,\mathcal{V}^r)}=0, \text{ } \forall r\in[0,1[
\end{equation}
\end{lemma}
The proof of this result is postponed to the appendix.

The next lemma is essentially, which is also proved in the appendix, is a rephrasing of the previous one with  $u_\eta$ instead of $u$. The difference is that we do not have anymore $\ue-\ues=0$ ins $\oes$, but we do have an estimate on it, from (\ref{hu_bornee_1}), which allows to pass to the limit as $\eta$ goes to $0$.
\begin{lemma}\label{convergence_Pueta_ueta}
\begin{equation}
\displaystyle\lim_{\sigma\to 0} \displaystyle\lim_{\eta\to 0} \Vert P^r_{\sigma}(u_{\eta})-u_{\eta} \Vert_{L^2(0,T,\mathcal{V}^r)}=0, \text{ } \forall r\in[0,1[
\end{equation}
\end{lemma}

\begin{lemma}\label{convergence_uetaPueta_uPu}
\begin{equation}
\displaystyle\lim_{\sigma\to 0}\displaystyle\lim_{\eta\to 0} \Vert \rho_{\eta}u_{\eta}\cdot P^r_{\sigma}(u_{\eta})-\rho u\cdot P^r_{\sigma}(u) \Vert_{L^1([0,T]\times\Omega)}=0, \text{ } \forall r\in]0,1[.
\end{equation}
\end{lemma}
\begin{proof}
Let $r\in]0,1[$ and $\sigma>0$. From Lemma \ref{inclusion_omega_s} there exists $\eta_0>0$ such that $\forall \eta<\eta_0$,
\begin{equation*}
\oes\subset\Omega_{s,\sigma/3}(t), \text{ }\forall t\in[0,T].
\end{equation*}
Let $\eta<\eta_0$. Arguing as in \cite{SanMarStaTuc02} we split $[0,T]$ in $N_T$ subintervals $I_k=[(k-1)\tau,k\tau]$, $\tau=T/N_T$, $k=1,..,N_T$. We choose $N_T$ large enough (depending on $\sigma$) such that
\begin{equation}\label{l3_incl_1}
\Omega_{s,\sigma/2}(k\tau)\subset\Omega_{s,\sigma}(t) \text{ and }\Omega_{s,\sigma/3}(t)\subset\Omega_{s,\sigma/2}(k\tau),\quad \forall t\in I_k, \text{ }\forall k=1,..,N_T.
\end{equation}
This is possible since $\Omega_{s,\sigma}(t)$ is moving with a rigid velocity field, $L^2$ in time. Thus the flow generated by this velocity field is continuous in time, and $\Omega_{s,\sigma}(t)$ is the image of $\Omega_{s,\sigma}^0$ through this map.
On each subinterval $I_k$, $k=1,..,N_T$, we consider the momentum equation
\begin{equation*}
\rho_{\eta}\dfrac{\partial u_{\eta}}{\partial t} +\rho_{\eta}(u_{\eta} \cdot\nabla) u_{\eta}-2\mu \dv(D(u_{\eta}))+\nabla p_{\eta}+\dfrac{1}{\eta}\re H_{\eta}(u_{\eta}-u_{\eta ,s})-\rho_{\eta}g=0.
\end{equation*}
Let us consider a test function $\xi$ vanishing outside $I_k$ and such that $\xi(.,t)\in \mathcal{K}^1_{\sigma/2}(k\tau)$ for $t\in I_k$. Since $\xi(.,t)$ is rigid on $\Omega_{s,\sigma/2}(k\tau)\supset \Omega_{s,\sigma/3}(t)\supset\oes$, Lemma \ref{propR1} yields
\begin{equation*}
\displaystyle\int_{I_k}\int_{\Omega} \left[\rho_{\eta}u_{\eta}\cdot \xi_t +(\rho_{\eta}u_{\eta}\otimes u_{\eta}-2\mu D(u_{\eta})):D(\xi)+\rho_{\eta}g\cdot\xi\right] \, dx dt =0.
\end{equation*} 
From bounds given by (\ref{bornerhoH}), (\ref{u_bornee_1}) and (\ref{u_bornee_2}) we derive the following estimates:
\begin{eqnarray*}
\left|\displaystyle\int_{I_k}\int_{\Omega}  D(u_{\eta}):D(\xi) \, dx dt\right| & \leq & \Vert D(u_{\eta}) \Vert_{L^2(I_k,L^2(\Omega))} \Vert D(\xi) \Vert_{L^2(I_k,L^2(\Omega))}\notag\\
& {\leq} & C \Vert \xi \Vert_{L^2(I_k,H^1_0(\Omega))}\le C \Vert \xi \Vert_{L^2(I_k,\mathcal{K}^1_{\sigma/2}(k\tau))}\notag\\
& \leq & C \Vert \xi \Vert_{L^4(I_k,\mathcal{K}^1_{\sigma/2}(k\tau))},
\end{eqnarray*}
\begin{eqnarray*}
\left|\displaystyle\int_{I_k}\int_{\Omega} (\rho_{\eta}u_{\eta}\otimes u_{\eta}):D(\xi) \, dx dt\right| & \leq & \displaystyle\int_{I_k} \Vert \rho_{\eta} \Vert_{L^{\infty}(\Omega)}\Vert u_{\eta}\otimes u_{\eta} \Vert_{L^2(\Omega)} \Vert D(\xi) \Vert_{L^2(\Omega)} dt\notag\\
& \leq & \Vert \rho_{\eta} \Vert_{L^{\infty}(I_k,L^{\infty}(\Omega))}\displaystyle\int_{I_k} \Vert u_{\eta} \Vert^2_{L^4(\Omega)}\Vert \xi \Vert_{H^1_0(\Omega)} dt\notag\\
& {\leq} & C \displaystyle\int_{I_k} \Vert u_{\eta} \Vert^{\frac{1}{2}}_{L^2(\Omega)}\Vert \nabla u_{\eta} \Vert^{\frac{3}{2}}_{L^2(\Omega)}\Vert \xi \Vert_{H^1_0(\Omega)} dt\notag\\
& \leq & C \Vert u_{\eta} \Vert^\frac12_{L^{\infty}(I_k,L^2(\Omega))}\displaystyle\int_{I_k} \Vert \nabla u_{\eta} \Vert^{\frac{3}{2}}_{L^2(\Omega)}\Vert \xi \Vert_{H^1_0(\Omega)} dt\notag\\
& \le & C {\Vert \nabla u_{\eta} \Vert^{\frac{2}{3}}_{L^2(I_k,L^2(\Omega))}}\Vert \xi \Vert_{L^4(I_k,H^1_0(\Omega))} \notag\\
& \leq & C \Vert \xi \Vert_{L^4(I_k,H^1_0(\Omega))}
\le C \Vert \xi \Vert_{L^4(I_k,\mathcal{K}^1_{\sigma/2}(k\tau))},
\end{eqnarray*}
and 
\begin{equation*}
\left|\displaystyle\int_{I_k}\int_{\Omega} \rho_{\eta}g\cdot \xi \, dx dt\right| \leq \Vert \rho_{\eta} \Vert_{L^{\infty}(I_k,L^{\infty}(\Omega))}\Vert \xi \Vert_{L^4(I_k,\mathcal{K}^1_{\sigma/2}(k\tau))} \leq C \Vert \xi \Vert_{L^4(I_k,\mathcal{K}^1_{\sigma/2}(k\tau))}.
\end{equation*}
Collecting terms we get
\begin{equation*}
\left|\displaystyle\int_{I_k}\int_{\Omega} \rho_{\eta}u_{\eta}\cdot  \xi_t \, dx dt \right|\leq C \Vert \xi \Vert_{L^4(I_k,\mathcal{K}^1_{\sigma/2}(k\tau))}.
\end{equation*}
As $\xi(.,t)\in \mathcal{K}^1_{\sigma/2}(k\tau)$, $\xi_t(.,t)\in \mathcal{K}^1_{\sigma/2}(k\tau)\subset \mathcal{K}^0_{\sigma/2}(k\tau)$ and we have
\begin{eqnarray*}
|\langle \rho_{\eta}u_{\eta}, \xi_t \rangle| & = & |\langle \rho_{\eta}u_{\eta}, P^0_{\sigma/2}(k\tau)\xi_t \rangle|\\
& =& |\langle P^0_{\sigma/2}(k\tau)(\rho_{\eta}u_{\eta}), \xi_t \rangle|\\
& = & |\langle \frac{d}{dt}P^0_{\sigma/2}(k\tau)(\rho_{\eta}u_{\eta}), \xi \rangle|.
\end{eqnarray*}
Therefore
\begin{equation*}
\left|\displaystyle\int_{I_k}\int_{\Omega} \frac{d}{dt}P^0_{\sigma/2}(k\tau)(\rho_{\eta}u_{\eta})\cdot  \xi \, dx dt \right|\leq C \Vert \xi \Vert_{L^4(I_k,\mathcal{K}^1_{\sigma/2}(k\tau))},
\end{equation*}
which means that
\begin{equation}\label{pu_bornee_1}
\frac{d}{dt}P^0_{\sigma/2}(k\tau)(\rho_{\eta}u_{\eta}) \text{ bounded in } L^{\frac{4}{3}}(I_k,(\mathcal{K}^1_{\sigma/2}(k\tau))^*).
\end{equation}
Moreover $\rho_{\eta}u_{\eta}$ is bounded in $L^2(I_k,L^2(\Omega))$, 
\begin{equation}\label{pu_bornee_2}
P^0_{\sigma/2}(k\tau)(\rho_{\eta}u_{\eta}) \text{ is bounded in } L^2(I_k,\mathcal{K}^0_{\sigma/2}(k\tau)).
\end{equation}
Since $\mathcal{K}^0_{\sigma/2}(k\tau)\subset (\mathcal{K}^r_{\sigma/2}(k\tau))^*$ compactly for $r>0$, and $(\mathcal{K}^r_{\sigma/2}(k\tau))^*\subset (\mathcal{K}^1_{\sigma/2}(k\tau))^*$ continuously for $r<1$, by the Aubin-Simon Lemma (see e.g. \cite{Boyer2006}, p. 98) with (\ref{pu_bornee_1}) and (\ref{pu_bornee_2}), we obtain the relative compactness of the sequence $\left(P^0_{\sigma/2}(k\tau)(\rho_{\eta}u_{\eta})\right)$ in $L^2(I_k,(\mathcal{K}v^r_{\sigma/2}(k\tau))^*)$ for all $r\in]0,1[$.\\
From (\ref{u_cv_1}) we deduce
\begin{equation}\label{pu_cv_1}
\displaystyle\lim_{\eta\to 0}P^0_{\sigma/2}(k\tau)(\rho_{\eta}u_{\eta}) = P^0_{\sigma/2}(k\tau)\rho u \text{ in } L^2(I_k,(\mathcal{K}^r_{\sigma/2}(k\tau))^*) \text{ strong, }\forall r\in]0,1[.
\end{equation}
Since from (\ref{l3_incl_1}), we have
\begin{equation}\label{idem}
P^0_{\sigma/2}(k\tau)P^r_{\sigma}(t)=P^r_{\sigma}(t) \text{ }\forall t\in I_k \text{ }\forall r\in]0,1[,
\end{equation}
and we can write
\begin{eqnarray*}
\displaystyle\int_{I_k} \langle \rho_{\eta}u_{\eta}, P^r_{\sigma}(t)(u_{\eta}) \rangle_{L^2(\Omega)} dt & {=} & \displaystyle\int_{I_k} \langle \rho_{\eta}u_{\eta}, P^0_{\sigma/2}(k\tau)P^r_{\sigma}(t)u_{\eta} \rangle_{L^2(\Omega)} dt\\
& = & \displaystyle\int_{I_k} \langle P^0_{\sigma/2}(k\tau)(\rho_{\eta}u_{\eta}), P^r_{\sigma}(t)u_{\eta} \rangle_{L^2(\Omega)} dt\\
& = & \displaystyle\int_{I_k} \langle P^0_{\sigma/2}(k\tau)(\rho_{\eta}u_{\eta}), P^r_{\sigma}(t)u_{\eta} \rangle_{(\mathcal{K}^r_{\sigma/2})^*,\mathcal{K}^r_{\sigma/2}} dt.\\
\end{eqnarray*}
The sequence $(u_{\eta})$ is bounded in $L^2(0,T,\mathcal{V}^r)$ for all $r\in]0,1[$,  therefore $(P^r_{\sigma}(t)u_{\eta})$ is bounded in $L^2(0,T,\mathcal{K}^r_{\sigma/2})$ for all $r\in]0,1[$.\\
Therefore there exists a subsequence of $P^r_{\sigma}(t)u_{\eta}$ still denoted $P^r_{\sigma}(t)u_{\eta}$ such that
\begin{equation}\label{pu_cv_2}
P^r_{\sigma}(t)u_{\eta} \rightharpoonup P^r_{\sigma}(t)u \text{ in } L^2(0,T,\mathcal{K}^r_{\sigma/2}) \text{ weak.}
\end{equation}
Passing to the limit in $\eta$ yields
\begin{eqnarray*}
\displaystyle\lim_{\eta\to 0}\int_{I_k} \langle \rho_{\eta}u_{\eta}, P^r_{\sigma}(t)(u_{\eta}) \rangle_{L^2(\Omega)} dt & {=} & \displaystyle\int_{I_k} \langle P^0_{\sigma/2}(k\tau)\rho u, P^r_{\sigma}(t)u \rangle_{L^2(\Omega)} dt\\
&= & \displaystyle\int_{I_k} \langle \rho u, P^0_{\sigma/2}(k\tau)P^r_{\sigma}(t)u \rangle_{L^2(\Omega)} dt\\
& {=} & \displaystyle\int_{I_k} \langle \rho u, P^r_{\sigma}(t)u \rangle_{L^2(\Omega)} dt.
\end{eqnarray*}
Summing over $k=1,..,N_T$, we finally obtain
\begin{equation*}
\displaystyle\lim_{\eta\to 0} \Vert \rho_{\eta}u_{\eta}\cdot P^r_{\sigma}(u_{\eta})-\rho u\cdot P^r_{\sigma}(u) \Vert_{L^1(Q)}=0,
\end{equation*}
which implies
\begin{equation*}
\displaystyle\lim_{\sigma\to 0}\displaystyle\lim_{\eta\to 0} \Vert \rho_{\eta}u_{\eta}\cdot P^r_{\sigma}(u_{\eta})-\rho u\cdot P^r_{\sigma}(u) \Vert_{L^1(Q)}=0.
\end{equation*}
\end{proof}
\noindent
We can now conclude to the strong convergence of $\ue$.

\noindent Let $\varepsilon>0$. From Lemma \ref{convergence_Pu_u},
\begin{equation*}
\exists \sigma_0 >0,\text{ } \forall \sigma < \sigma_0,\text{ } \Vert P^r_{\sigma}(u)-u \Vert_{L^2(Q)} < \varepsilon.
\end{equation*}
From Lemma \ref{convergence_Pueta_ueta},
\begin{equation*}
\exists \sigma_0 >0,\text{ } \forall \sigma < \sigma_0,\text{ } \exists \eta_0 >0,\text{ } \forall \eta < \eta_0,\text{ } \Vert P^r_{\sigma}(u_{\eta})-u_{\eta} \Vert_{L^2(Q)} < \varepsilon,
\end{equation*}
and by Lemma \ref{convergence_uetaPueta_uPu},
\begin{equation*}
\exists \sigma_0 >0,\text{ } \forall \sigma < \sigma_0,\text{ } \exists \eta_0 >0,\text{ } \forall \eta < \eta_0,\text{ } \Vert \rho_{\eta}u_{\eta}P^r_{\sigma}(u_{\eta})-\rho uP^r_{\sigma}(u) \Vert_{L^1(Q)} < \varepsilon.
\end{equation*}
We therefore get from (\ref{decomp}) (up to the extraction of a subsequence)
\begin{equation*}
\exists \eta_0 >0,\text{ } \forall \eta < \eta_0,\text{ } \displaystyle\int_0^T\int_{\Omega} \vert u_{\eta}-u \vert^2 \, dx dt < C\varepsilon
\end{equation*}
which means that (still up to a subsequence)
\begin{equation}\label{u_cv_4}
u_{\eta} \to u \text{ in } L^2(Q) \text{ strong.}
\end{equation}
Classically, we also obtain from (\ref{rho_cv_1})
\begin{equation}\label{u_cv_5}
\rho_{\eta}u_{\eta} \rightharpoonup \rho u \text{ in } L^2(Q) \text{ weak}.
\end{equation}
\subsection{Passing to the limit}
Let us now prove that as $\eta$ goes to zero, a subsequence of $(\ue,\re)$ converges toward $(u,\rho)$ solution of the weak formulation (\ref{forme_faible}). Indeed~:
We have proved that
$
\rho_{\eta} \rightharpoonup \rho  \text{ in } L^{\infty}(0,T,L^{\infty}(\Omega)) \text{ weak}*.$ Therefore
$$\rho\in L^{\infty}(Q).$$
We have proved that $u_{\eta} \rightharpoonup u \text{ in } L^2(0,T,V)$ weak, that $\sqrt{\rho_{\eta}}u_{\eta}$ is bounded in $L^{\infty}(0,T,L^2(\Omega))$ and $\rho_{\eta}$ bounded from above and below in  $L^{\infty}(0,T,L^{\infty}(\Omega))$.
This implies that $u_{\eta} \text{ bounded in } L^{\infty}(0,T,L^2(\Omega))$, thus its weak limit belongs to  $L^{\infty}(0,T,L^2(\Omega))\cap L^2(0,T,V)$:
$$u\in L^{\infty}(0,T,L^2(\Omega))\cap L^2(0,T,V).$$
From Lemma \ref{u_limite_rigide}, we have
$
H u=H u_s=H(u_G+\omega\times r)
$
where $H$ is the characteristic  function of $\Omega_s(t)$. Thus
$$u(t)\in \mathcal{K}(t).$$
Using compactness results of DiPerna-Lions we already obtained that $\rho$ and $H$ are solutions of transport equations with $u$ and $u_s$ as velocities. For $H$ this means that for all $\psi\in\mathcal{C}^1(Q)$ with $\psi(T)=0$, 
\begin{eqnarray*}
\displaystyle\int_0^T\int_{\Omega}H\dfrac{\partial \psi}{\partial t} +H u_s\cdot\nabla\psi \, dx dt+ \displaystyle\int_{\Omega} H^0 \psi(0) \, dx=0.
\end{eqnarray*}
As from Lemma \ref{u_limite_rigide}, $Hu_s=Hu$, $H$ is also solution of 
\begin{eqnarray*}
\displaystyle\int_0^T\int_{\Omega}H\dfrac{\partial \psi}{\partial t} +H u\cdot\nabla\psi \, dx dt+ \displaystyle\int_{\Omega} H^0 \psi(0) \, dx=0.
\end{eqnarray*}
In other terms $H$, like $\rho$ satisfies a transport equation with velocity $u$.

\noindent Let us finally check that $u$ satisfies the momentum equation.

\noindent  Let $\sigma>0$. If  $\xi_\sigma\in H^1(Q)\cap L^2(0,T;\mathcal{K}_\sigma^1(t))$, from (\ref{ns1penal}) and (\ref{rhopenal}) we get
\begin{eqnarray*}
\displaystyle\int_{\Omega}\left[\dfrac{\partial (\rho_{\eta}u_{\eta})}{\partial t} +\dv(\rho_{\eta}u_{\eta} \otimes u_{\eta})-2\mu \dv(D(u_{\eta}))+\nabla p_{\eta}+\dfrac{1}{\eta}H_{\eta}\rho_{\eta}(u_{\eta}-u_{\eta ,s})-\rho_{\eta}g\right]\cdot\xi_\sigma \, dx=0.
\end{eqnarray*}
From Lemma \ref{inclusion_omega_s}, there exists $\eta_0$ such that $\eta<\eta_0$ implies:
\begin{eqnarray*}
\displaystyle\int_{\Omega}H_{\eta}\rho_{\eta}(u_{\eta}-u_{\eta ,s})\cdot\xi_\sigma \, dx=0.
\end{eqnarray*}
By integration by parts
\begin{eqnarray*}
\displaystyle\int_{\Omega}-2\mu \dv(D(u_{\eta}))\cdot\xi_\sigma \, dx=\displaystyle\int_{\Omega} 2\mu D(u_{\eta}):D(\xi_\sigma)\, dx,
\end{eqnarray*}
\begin{eqnarray*}
\displaystyle\int_{\Omega} \dv(\rho_{\eta}u_{\eta} \otimes u_{\eta})\cdot\xi_\sigma \, dx=\displaystyle\int_{\Omega} -(\rho_{\eta}u_{\eta} \otimes u_{\eta}):D(\xi_\sigma)\, dx,
\end{eqnarray*}
\begin{eqnarray*}
\displaystyle\int_{\Omega} \dfrac{\partial (\rho_{\eta}u_{\eta})}{\partial t}\cdot\xi_\sigma \, dx=\dfrac{d}{dt}\displaystyle\int_{\Omega} \rho_{\eta}u_{\eta}\cdot\xi_\sigma \, dx-\displaystyle\int_{\Omega} \rho_{\eta}u_{\eta}\cdot\dfrac{\partial \xi_\sigma}{\partial t} \, dx.
\end{eqnarray*}
As a result
\begin{eqnarray*}
\displaystyle\int_{\Omega} \left[\rho_{\eta}u_{\eta}\cdot\dfrac{\partial \xi_\sigma}{\partial t} +\left( \rho_{\eta}u_{\eta} \otimes u_{\eta}- 2\mu D(u_{\eta})\right):D(\xi_\sigma)+ \rho_{\eta}g\cdot\xi_\sigma \right] \, dx=\dfrac{d}{dt}\displaystyle\int_{\Omega} \rho_{\eta}u_{\eta}\cdot\xi_\sigma \, dx.
\end{eqnarray*}
We have already established that $u_{\eta} \rightharpoonup u \text{ in } L^2(0,T,H^1_0(\Omega)) \text{ weak}$, 
$u_{\eta} \to u \text{ in } L^2(0,T,L^2(\Omega)) \text{ strong}$, 
$\rho_{\eta}u_{\eta} \rightharpoonup \rho u \text{ in } L^2(0,T,L^2(\Omega)) \text{ weak}$, 
and $\rho_{\eta} \to \rho \text{ in } L^2(0,T,L^2(\Omega)) \text{ strong}$. Letting 
$\eta$ goes to zero, we thus obtain \begin{eqnarray*}
\displaystyle\int_{\Omega} \left[\rho u\cdot\dfrac{\partial \xi_\sigma}{\partial t} +\left( \rho u \otimes u- 2\mu D(u)\right):D(\xi_\sigma)+ \rho g\cdot\xi_\sigma \right] \, dx=\dfrac{d}{dt}\displaystyle\int_{\Omega} \rho u\cdot \xi_\sigma \, dx.
\end{eqnarray*}
which corresponds to the weak formulation (\ref{forme_faible}). This holds for any $\xi_\sigma\in H^1(Q)\cap L^2(0,T;K_\sigma^1(t))$, for arbitrary $\sigma>0$, and, since the time interval has been chosen to guarantee that there is no contact with the boundary, by Proposition 4.3 of \cite{SanMarStaTuc02}, for all $\xi\in H^1(Q)\cap L^2(0,T;K(t))$. This ends the proof of theorem \ref{cvresult}.
\section{Numerical simulations}
We give here a few numerical illustrations of the penalization method. We only sketch the numerical discretization and we refer to \cite{Bost2008} for a more detailed description and further numerical results.

We choose a time-step $\Delta t$ and denote by a superscript $n$ discretization of all quantities at time $t_n=n \Delta t$. For each time integration, we split the penalization model  (\ref{ns1penal})-(\ref{hpenal}) as follows 
\begin{itemize}
\item We solve the following variable density flow problem and obtain:
\begin{equation*}
\widetilde{u}=u^n-\Delta t(u^n.\nabla)u^n+ \dfrac{\Delta t}{\rho^n} \mu\Delta \widetilde{u}-\dfrac{\Delta t}{\rho^n}\nabla p^{n+1}+\Delta t g
\end{equation*}
with
\begin{equation*}
\rho^n=\rho_s\,H^n+\rho_f\,\left(1-H^n\right)
\end{equation*}
\item We compute the rigid velocity the rigid velocity corresponding to $\tilde u$: 
$$u_s=u_G+w \times r^n,\quad u_G=\dfrac{\displaystyle\int_\Omega \rho^n\, \widetilde{u}\, H^n \, dx}{\displaystyle\int_\Omega \rho^n\, H^n \, dx},\quad  \omega=J^{-1}\int_\Omega \rho^n (r^n\times \widetilde{u})\, H^n \, dx$$
\item We penalize the flow with this rigid velocity inside the solid body:
$$\dfrac{u^{n+1}-\widetilde{u}}{\Delta t}=\dfrac{1}{\eta}H^n(u_s-u^{n+1}) \Leftrightarrow u^{n+1}=\dfrac{\widetilde{u}+\frac{\Delta t}{\eta}H^n u_s}{1+\frac{\Delta t}{\eta}H^n}$$
\item We finally advect  the solid with the rigid velocity:
$$ H^{n+1}=H^n-\Delta t\, u_s\cdot\nabla H^n$$
\end{itemize}
Note that we have used an implicit time discretization of the velocity penalization, which allows to use a very small penalization parameter $\eta$. Using an explicit method would require this value no to be smaller than $\Delta t$. It can indeed be checked that a explicit method with $\eta=\Delta t$ essentially amounts to the projection method \cite{Pat01}. We will see below that using smaller values of $\eta$ together with an implicit scheme has a significant effect on the accuracy of the method. 

In order to numerically validate the penalization method, we consider the case  of the sedimentation of a rigid cylinder in two dimensions (see \cite{GloPan01}, \cite{Coquerelle2008}).
The domain $\Omega=[0,2]\times[0,6]$ is filled with an incompressible viscous fluid initially at rest, of 
density $\rho_f=1$ and viscosity $\mu_f=0.01$. The rigid cylinder of radius $R=0.125$ and density $\rho_s=1.5$ is initially centered in 
$(1,4)$, and we apply the gravity force $g=-980$.\\ 
In order to verify how the rigid constraint is satisfied in the solid, we monitor at time $t=0.1$ the $L^2$-norm of 
the discrete deformation tensor defined by:
\begin{equation*}
\Vert D(u)\Vert^2_{L^2(\Omega_s(t))}:=\displaystyle\sum_{ij,\phi_{ij}<0}\left(D_{11}^2(u_{ij})+2D_{12}^2(u_{ij})+D_{22}^2(u_{ij})\right)(\Delta x)^2
\end{equation*}
We fix $\Delta x=1/256$ and $\Delta t=10^{-4}$, and compute this norm for values of $\eta$ from $10^{-4}$ to $10^{-12}$, at $t=0.1$.\\
The results presented in table \ref{erreurD01-512} indicate  a convergence of the method with order 1 in $\eta$.\\
\begin{table}[!h]
\begin{center}
\begin{tabular}{|c|c|c|}
\hline
$\eta$ & $\Vert D(u)\Vert_{L^2(\Omega_s(t))}$ & $\alpha$ for $O(\eta^{\alpha})$  \rule[-11pt]{0pt}{24pt}\\ \hline
$10^{-4}$ & $4.30838$ & - \rule[-11pt]{0pt}{24pt}\\ 
$10^{-6}$ & $3.84749\times 10^{-2}$ & 1.0247 \rule[-11pt]{0pt}{24pt}\\ 
$10^{-8}$ & $3.45379\times 10^{-4}$ & 1.0234 \rule[-11pt]{0pt}{24pt}\\ 
$10^{-10}$ & $3.81643\times 10^{-6}$ & 0.9783 \rule[-11pt]{0pt}{24pt}\\ 
$10^{-12}$ & $3.79832\times 10^{-8}$ & 1.001 \rule[-11pt]{0pt}{24pt}\\
\hline
\end{tabular}
\end{center}
\caption{Sedimentation of a two dimensional cylinder. Errors on $\Vert D(u)\Vert_{L^2(\Omega_s(t))}$ and convergence orders at $t=0.1$ for $\Delta x=1/256$ and $\Delta t=10^{-4}$.}
\label{erreurD01-512} 
\end{table}
In figure \ref{profil-vitesse-eta} we show the profiles of the vertical velocity for several values of $\eta$, corresponding to a cross section at the center of the cylinder. 
We can observe that below $\eta=10^{-8}$ one may consider that we obtained converged velocity results. 
By taking $\eta=10^{-4}$, that is $\eta= \Delta t$, along with an explicit treatment of the penalization term, we get the projection method. As far as precision is concerned, one can note the benefit of using larger penalization parameters combined with an implicit time discretization of the penalization term.  


\begin{figure}[!h]
\centerline{\includegraphics[width=8cm]{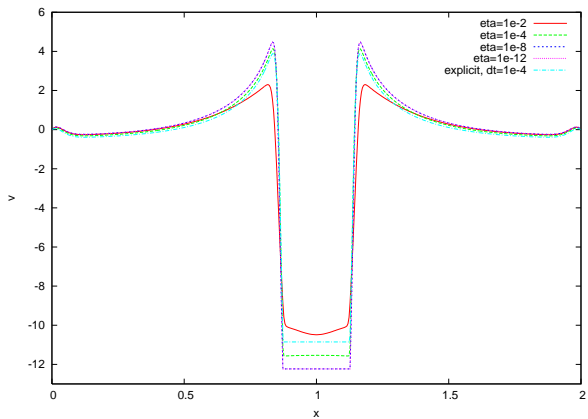}}
\caption{Sedimentation of a two-dimensional cylinder, for $\Delta x=1/256$ and $\Delta t=10^{-4}$. Vertical velocity in an horizontal cross-section through the center of the cylinder at $t=0.1$ for several values of the penalization parameter. }
\label{profil-vitesse-eta}
\end{figure}
\section{Conclusion}
We have presented and analyzed a penalization method that extends the method of \cite{AngBruFab99} to the case of a rigid body moving freely in an incompressible fluid. The proof is based on compactness arguments. Numerical illustrations have been provided to illustrate our convergence result. The benefit of using very large penalization parameters combined with an implicit time discretization of the penalization term, compared to the projection method \cite{ShaPat05} which corresponds to a particular explicit time discretization for the penalized equation,  has been demonstrated.

While this was not our primary goal, an outcome of our convergence study is an existence result for a weak formulation of the coupling between a rigid solid and a fluid. Let us shortly discuss how this result compares with existing ones \cite{GraMad98,DejEst99,ConSanTuc00, SanMarStaTuc02}. In \cite{GraMad98} local in time existence and uniqueness of strong solutions was proved. The Eulerian approach was developed in \cite{DejEst99} where global in time existence of weak solutions was proved in dimension $2$, without collisions. In the three-dimensional case, to our knowledge  only local in time existence of weak solutions was obtained, since $L^2$ regularity of the time derivative of velocity was required  (and therefore global existence would imply global existence of strong solutions). In \cite{ConSanTuc00,GunLeeSer00} the existence of global weak solution in three dimensions for one ball shaped solid, with possible collision with the boundary, was proved. In \cite{SanMarStaTuc02} the existence of global weak solutions for several rigid bodies with collisions was proved in dimension $2$. Our results prove the existence of global in time weak solutions in three dimensions, before collision. By contrast with \cite{ConSanTuc00,GunLeeSer00}, this result can easily be generalized to the case of several bodies by introducing indicator functions, rigid velocities and penalization terms corresponding to each body. To our knowledge, the existence of global in time weak solutions for several bodies with collisions is an open problem in three dimensions. 


\section{Appendix} 
This section si devoted to the proof of some technical lemmas that were used in section 3.

\paragraph{Proof of Lemma 3.3}
Let $\phi(t)\in L^2(\Omega\backslash\overline{\Omega_{s,\sigma}}(t))$, and $(v(t),q(t))\in H^1(\Omega\backslash\overline{\Omega_{s,\sigma}}(t))\times L^2(\Omega\backslash\overline{\Omega_{s,\sigma}}(t))$ solution of the Stokes problem
\begin{equation*}
\begin{cases}
-\Delta v(t) + \nabla q(t)=\phi(t) &\text{on }\Omega\backslash\overline{\Omega_{s,\sigma}}(t),\\
\dv v(t)=0 &\text{on }\Omega\backslash\overline{\Omega_{s,\sigma}}(t),\\
v(t)=0 &\text{on }\partial\left(\Omega\backslash\overline{\Omega_{s,\sigma}}(t)\right).
\end{cases}
\end{equation*}
Since we assumed a $\mathcal{C}^2$ regularity on $\Omega\backslash\Omega_{s}^0$ this regularity is conserved through rigid motion and, for $\sigma$ small enough (say $\sigma<\sigma_0$ for some $\sigma_0>0$) to $\Omega\backslash\Omega_{s,\sigma}(t)$. The regularity results of Agmon-Douglis-Nirenberg on the linear Stokes problem (see \cite{Temam1979}, prop. 2.3. p. 35) give
\begin{equation*}
(v(t),q(t))\in H^2(\Omega\backslash\overline{\Omega_{s,\sigma}}(t))\times H^1(\Omega\backslash\overline{\Omega_{s,\sigma}}(t))
\end{equation*}
and there exists $C>0$ such that
\begin{equation*}
\Vert v(t) \Vert_{H^2(\Omega\backslash\overline{\Omega_{s,\sigma}}(t))}+\Vert q(t) \Vert_{H^1(\Omega\backslash\overline{\Omega_{s,\sigma}}(t))} \leq C \Vert \phi(t) \Vert_{L^2(\Omega\backslash\overline{\Omega_{s,\sigma}}(t))}.
\end{equation*}
Note that, with our definition of $T$, the constant $C$, which depends on the geometry of the domain boundary, can be taken independent of $t\in [0,T]$ and $\sigma$, provided $\sigma_0$ is taken small enough. We can then write
\begin{eqnarray*}
\displaystyle\int_{\Omega\backslash\overline{\Omega_{s,\sigma}}(t)} w(t)\cdot \phi(t) \, dx & = &  -\displaystyle\int_{\Omega\backslash\overline{\Omega_{s,\sigma}}(t)} w(t) \cdot\Delta v(t) \, dx + \displaystyle\int_{\Omega\backslash\overline{\Omega_{s,\sigma}}(t)} w(t) \cdot\nabla q(t) \, dx\\
& = &  -\displaystyle\int_{\partial(\Omega\backslash\overline{\Omega_{s,\sigma}})(t)} w(t) \cdot\dfrac{\partial v(t)}{\partial n} ds+\displaystyle\int_{\Omega\backslash\overline{\Omega_{s,\sigma}}(t)} \nabla w(t) \cdot\nabla v(t) \, dx\\
& & +  \displaystyle\int_{\partial(\Omega\backslash\overline{\Omega_{s,\sigma}})(t)} w(t) q(t)\cdot n ds-\displaystyle\int_{\Omega\backslash\overline{\Omega_{s,\sigma}}(t)} \dv w(t) q(t) \, dx\\
& = &  -\displaystyle\int_{\partial\Omega_{s,\sigma}(t)} w(t) \cdot\dfrac{\partial v(t)}{\partial n} ds+\displaystyle\int_{\partial(\Omega\backslash\overline{\Omega_{s,\sigma}})(t)} \dfrac{\partial w(t)}{\partial n}\cdot v(t)  ds\\
&  & -\displaystyle\int_{\Omega\backslash\overline{\Omega_{s,\sigma}}(t)} \Delta w(t)\cdot v(t) \, dx + \displaystyle\int_{\partial\Omega_{s,\sigma}(t)} w(t) q(t)\cdot n ds\\
& = &  -\displaystyle\int_{\partial\Omega_{s,\sigma}(t)} w(t) \cdot\dfrac{\partial v(t)}{\partial n} ds+\displaystyle\int_{\Omega\backslash\overline{\Omega_{s,\sigma}}(t)} v(t) \cdot\nabla p(t) \, dx \\
&  & + \displaystyle\int_{\partial\Omega_{s,\sigma}(t)} w(t) q(t)\cdot n ds
\end{eqnarray*}
The integral of $v\cdot\nabla p$ vanishes since $v$ is divergence-free and vanishes on $\partial\left(\Omega\backslash\overline{\Omega_{s,\sigma}}(t)\right)$. Then using classical trace theorems in Sobolev spaces we get
\begin{eqnarray*}
\displaystyle\int_{\Omega\backslash\overline{\Omega_{s,\sigma}}(t)} w(t) \cdot\phi(t) \, dx& = &  -\displaystyle\int_{\partial\Omega_{s,\sigma}(t)} g(t) \cdot\dfrac{\partial v(t)}{\partial n} ds + \displaystyle\int_{\partial\Omega_{s,\sigma}(t)} g(t) q(t)\cdot n ds\\
& \leq & \Vert g(t) \Vert_{L^2(\partial\Omega_{s,\sigma}(t))} \left( \Vert \nabla v(t) \Vert_{L^2(\partial\Omega_{s,\sigma}(t))}+\Vert q(t) \Vert_{L^2(\partial\Omega_{s,\sigma}(t))}\right)\\
& \leq &C \Vert g(t) \Vert_{L^2(\partial\Omega_{s,\sigma}(t))} \left( \Vert v(t) \Vert_{H^2(\Omega\backslash\overline{\Omega_{s,\sigma}}(t))}+\Vert q(t) \Vert_{H^1(\Omega\backslash\overline{\Omega_{s,\sigma}}(t))}\right)\\
& \leq & C \Vert g(t) \Vert_{L^2(\partial\Omega_{s,\sigma}(t))} \Vert \phi(t) \Vert_{L^2(\Omega\backslash\overline{\Omega_{s,\sigma}}(t))}\\
& \leq & C {\Vert u(t) \Vert}^{\frac{1}{2}}_{L^2(\Omega_{s,\sigma}(t))} {\Vert \nabla u(t) \Vert}^{\frac{1}{2}}_{L^2(\Omega_{s,\sigma}(t))} \Vert \phi(t) \Vert_{L^2(\Omega\backslash\overline{\Omega_{s,\sigma}}(t))}.\\
\end{eqnarray*}
This proves the assertion.

\paragraph{Proof of Lemma 3.6}\mbox{}\\
\textbf{Step 1:} We first show how to construct for a.e. $t\in]0,T[$ a function $v_{\sigma}(.,t)\in \mathcal{K}^r_{\sigma}(t)$ such that
\begin{equation*}
\displaystyle\lim_{\sigma\to 0} \Vert v_{\sigma}(.,t)-u(.,t) \Vert_{\mathcal{V}^r}=0 \text{ a.e. on }]0,T[.
\end{equation*}
Let $\sigma>0$ and $v_{\sigma}(.,t)$ such that
\begin{equation*}
\begin{cases}
-\Delta v_{\sigma}(.,t)+\nabla p(.,t)=-\Delta u(.,t) &\text{on }\Omega\backslash\overline{\Omega_{s,\sigma}}(t),\\
\div  v_{\sigma}(.,t)=0&\text{on }\Omega\backslash\overline{\Omega_{s,\sigma}}(t),\\
v_{\sigma}(.,t)=u_{s}(.,t)&\text{on }\partial\Omega_{s,\sigma}(t),\\
v_{\sigma}(.,t)=0&\text{on }\partial\Omega,
\end{cases}
\end{equation*}
where 
\begin{equation*}
u_s=\dfrac{1}{M}\displaystyle\int_\Omega \rho u\, H \, dx
+\left(J^{-1}\,\displaystyle\int_\Omega \rho (r\times u) H \, dx\right)\times r.
\end{equation*}
By lemma \ref{u_limite_rigide}, $u(.,t)=u_s(.,t)$ on $\Omega_{s}(t)$. Extending 
$v_{\sigma}(.,t)$ by $u_{s}(.,t)$ in $\Omega_{s,\sigma}(t)$, we have $v_{\sigma}(.,t)\in \mathcal{K}^r_{\sigma}(t)$.
We set $e_{\sigma}(.,t)=v_{\sigma}(.,t)-u(.,t)$. It satisfies
\begin{equation*}
\begin{cases}
-\Delta e_{\sigma}(.,t)+\nabla p(.,t)=0 &\text{on }\Omega\backslash\overline{\Omega_{s,\sigma}}(t),\\
\div  e_{\sigma}(.,t)=0&\text{on }\Omega\backslash\overline{\Omega_{s,\sigma}}(t),\\
e_{\sigma}(.,t)=u_{s}(.,t)-u(.,t) &\text{on }\partial\Omega_{s,\sigma}(t),\\
e_{\sigma}(.,t)=0 &\text{on }\partial\Omega.
\end{cases}
\end{equation*}
We extend $e_{\sigma}(.,t)$ by $u_{s}(.,t)-u(.,t)$ in $\Omega_{s,\sigma}(t)$, so that $e_{\sigma}(.,t)=0$ in $\Omega_s(t)$.\\
We claim that
\begin{equation}\label{limite_l2_u}
\displaystyle\lim_{\sigma\to 0} \Vert e_{\sigma}(.,t) \Vert_{L^2(\Omega)}=0  \text{ a.e. on }]0,T[.
\end{equation}
In $\Omega_{s}(t)$ $e_{\sigma}(.,t)=0$, thus
\begin{equation*}
{\Vert e_{\sigma}(.,t) \Vert}^2_{L^2(\Omega)}={\Vert e_{\sigma}(.,t) \Vert}^2_{L^2(\Omega_{s,\sigma}(t)\backslash\overline{\Omega_{s}}(t))}+{\Vert e_{\sigma}(.,t) \Vert}^2_{L^2(\Omega\backslash\overline{\Omega_{s,\sigma}}(t))}.
\end{equation*}
Since $\Omega_{s,\sigma}(t)\backslash\overline{\Omega_{s}}(t)$ has width $2\sigma$, from the proof of lemma 5.10 of \cite{Fujita70} we have a.e. on $]0,T[$,
\begin{eqnarray*}
{\Vert e_{\sigma}(.,t) \Vert}^2_{L^2(\Omega_{s,\sigma}(t)\backslash\overline{\Omega_{s}}(t))} &\leq & C \left({\Vert e_{\sigma}(.,t) \Vert}^2_{L^2(\partial\Omega_{s}(t))}+\sigma^2 {\Vert \nabla e_{\sigma}(.,t) \Vert}^2_{L^2(\Omega_{s,\sigma}(t)\backslash\overline{\Omega_{s}}(t))}\right)\\
& \leq & C \left({\Vert e_{\sigma}(.,t) \Vert}_{L^2(\Omega_{s}(t))}{\Vert \nabla e_{\sigma}(.,t) \Vert}_{L^2(\Omega_{s}(t))}+\sigma^2 {\Vert \nabla e_{\sigma}(.,t) \Vert}^2_{L^2(\Omega_{s,\sigma}(t)\backslash\overline{\Omega_{s}}(t))}\right)\\
&=&C\sigma^2 {\Vert \nabla e_{\sigma}(.,t) \Vert}^2_{L^2(\Omega_{s,\sigma}(t)\backslash\overline{\Omega_{s}}(t))}
\end{eqnarray*}
Next, as  
$e_{\sigma}(.,t)=u_{s}(.,t)-u(.,t)$ in $\Omega_{s,\sigma}(t)$
and $u(.,t)$ and $u_{s}(.,t)$ are in $H^1_0(\Omega)$, we get
${\Vert \nabla e_{\sigma}(.,t) \Vert}_{L^2(\Omega_{s,\sigma}(t)\backslash\overline{\Omega_{s}}(t))}
\le{\Vert \nabla e_{\sigma}(.,t) \Vert}_{L^2(\Omega_{s,\sigma}(t))}\le{\Vert \nabla e_{\sigma}(.,t) \Vert}_{L^2(\Omega)} \leq C$, where $C$ is independent of $\sigma$. This
gives \begin{equation*}
{\Vert e_{\sigma}(.,t) \Vert}_{L^2(\Omega_{s,\sigma}(t))} = {\Vert e_{\sigma}(.,t) \Vert}_{L^2(\Omega_{s,\sigma}(t)\backslash\overline{\Omega_{s}}(t))}
\leq C \sigma.
\end{equation*}
By Lemma\;\ref{stokes} we thus get $$
{\Vert e_{\sigma}(.,t) \Vert}^2_{L^2(\Omega\backslash\overline{\Omega_{s,\sigma}}(t))}{\leq} C {\Vert e_{\sigma}(.,t) \Vert}_{L^2(\Omega_{s,\sigma}(t))} {\Vert \nabla e_{\sigma}(.,t) \Vert}_{L^2(\Omega_{s,\sigma}(t))} \le C\sigma.
$$
Collecting the above estimates, we conclude that 
\begin{equation*}
\displaystyle\lim_{\sigma\to 0} {\Vert e_{\sigma}(.,t) \Vert}^2_{L^2(\Omega)} =0.
\end{equation*}
In order to prove that this convergence also holds in $\mathcal{V}^r$ we first note that
\begin{equation}\label{borne_h1_u}
\Vert e_{\sigma}(.,t) \Vert_{H^1(\Omega)}\leq C \text{ a.e. on }]0,T[\\
\end{equation}
as is readily seen from estimates on the Stokes problem verified by $e_\sigma$. By interpolation (see e.g. \cite{adams}, p. 135), we obtain
\begin{equation}
\Vert e_{\sigma}(.,t) \Vert_{\mathcal{V}^r} \leq {\Vert e_{\sigma}(.,t) \Vert}^{1-r}_{L^2(\Omega)} {\Vert e_{\sigma}(.,t) \Vert}^{r}_{H^1(\Omega)}\label{interp}
\end{equation}
and due to (\ref{limite_l2_u}) and (\ref{borne_h1_u}),
\begin{equation}\label{limite_hs_u}
\displaystyle\lim_{\sigma\to 0} \Vert e_{\sigma}(.,t) \Vert_{\mathcal{V}^r}=0 \text{   } \forall r\in[0,1[ \text{ a.e. on }]0,T[.
\end{equation}
\textbf{Step 2:} By definition of $P^r_{\sigma}$, 
\begin{equation*}
\Vert P^r_{\sigma}u(.,t)-u(.,t) \Vert_{\mathcal{V}^r}\leq\Vert v_{\sigma}(.,t)-u(.,t) \Vert_{\mathcal{V}^r}
\end{equation*}
thus the pointwise convergence on $v_\sigma$ we just obtained implies
\begin{equation}
\displaystyle\lim_{\sigma\to 0} \Vert P^r_{\sigma}u(.,t)-u(.,t) \Vert_{\mathcal{V}^r}=0 \text{ a.e. on }]0,T[.
\end{equation}
\textbf{Step 3:}  $f_{\sigma}:t\mapsto\Vert P^r_{\sigma}u(.,t)-u(.,t) \Vert_{\mathcal{V}^r}$ is measurable on $[0,T]$ and since $0\in  \mathcal{K}^r_\sigma$,
\begin{eqnarray*}
{\Vert f_{\sigma} \Vert}^{\frac{2}{r}}_{L^{\frac{2}{r}}(0,T)} = \displaystyle\int_0^T  {\Vert P^r_{\sigma}u(.,t)-u(.,t) \Vert}^{\frac{2}{r}}_{\mathcal{V}^r} dt & {\leq} & \displaystyle\int_0^T  {\Vert u(.,t) \Vert}^{\frac{2}{r}}_{\mathcal{V}^r} dt\\
&{\leq} & C\displaystyle\int_0^T  {\Vert u(.,t) \Vert}^{\frac{2(1-r)}{r}}_{L^2(\Omega)}{\Vert u(.,t) \Vert}^{2}_{H^1(\Omega)} dt\\
& \leq & C {\Vert u \Vert}^{\frac{2(1-r)}{r}}_{L^{\infty}(0,T,L^2(\Omega))}{\Vert u\Vert}^{2}_{L^2(0,T,H^1_0(\Omega))} \\
&{\leq} & C.
\end{eqnarray*}
To summarize $f_\sigma$ verifies
\begin{equation*}
\begin{cases}
\displaystyle\lim_{\sigma\to 0} f_{\sigma}(t)=0 \text{ a.e. on }[0,T],\\
f_{\sigma} \text{ is measurable on } [0,T],\\
{\Vert f_{\sigma} \Vert}_{L^{\frac{2}{r}}(0,T)}\leq C\text{ with }r<1.
\end{cases}
\end{equation*}
Therefore, thanks to lemma \ref{egorov}, $\displaystyle\lim_{\sigma\to 0} {\Vert f_{\sigma} \Vert}_{L^{2}(0,T)}=0$, which means
\begin{equation*}
\displaystyle\lim_{\sigma\to 0} \Vert P^r_{\sigma}u-u \Vert_{L^2(0,T,\mathcal{V}^r)}=0.
\end{equation*}
\paragraph{Proof of Lemma 3.7}\mbox{}\\
\textbf{Step 1:} We construct for a.e. fixed $t\in[0,T]$ a function $v_{\eta\sigma}(.,t)\in \mathcal{K}^r_{\sigma}(t)$ such that
\begin{equation*}
\displaystyle\lim_{\sigma\to 0} \displaystyle\lim_{\eta\to 0} \Vert v_{\eta\sigma}(.,t)-u_{\eta}(.,t) \Vert_{\mathcal{V}^r}=0 \text{ a.e.  on }]0,T[.
\end{equation*}
Let $\sigma>0$ and  $v_{\eta\sigma}(.,t)$ solution of the following Stokes problem outside $\Omega_{s,\sigma}^\eta(t)$:
\begin{equation*}
\begin{cases}
-\Delta v_{\eta\sigma}(.,t)+\nabla p(.,t)=-\Delta u_{\eta}(.,t) &\text{on }\Omega\backslash\overline{\Omega^\eta_{s,\sigma}}(t),\\
\div  v_{\eta\sigma}(.,t)=0&\text{on }\Omega\backslash\overline{\Omega^\eta_{s,\sigma}}(t),\\
v_{\eta\sigma}(.,t)=u_{\eta,s}(.,t) &\text{on }\partial\Omega^\eta_{s,\sigma}(t),\\
v_{\eta\sigma}(.,t)=0 &\text{on }\partial\Omega.
\end{cases}
\end{equation*}
Extending $v_{\eta\sigma}(.,t)$ by $u_{\eta,s}(.,t)$ in $\Omega^\eta_{s,\sigma}(t)$, we have $v_{\eta\sigma}(.,t)\in \mathcal{K}^r_{\sigma}(t)$. We then  introduce $e_{\eta\sigma}(.,t)=v_{\eta\sigma}(.,t)-u_{\eta}(.,t)$. It verifies
\begin{equation*}
\begin{cases}
-\Delta e_{\eta\sigma}(.,t)+\nabla p(.,t)=0 &\text{on }\Omega\backslash\overline{\Omega^\eta_{s,\sigma}}(t),\\
\div  e_{\eta\sigma}(.,t)=0&\text{on }\Omega\backslash\overline{\Omega^\eta_{s,\sigma}}(t),\\
e_{\eta\sigma}(.,t)=u_{\eta,s}(.,t)-u_{\eta}(.,t) &\text{on }\partial\Omega^\eta_{s,\sigma}(t),\\
e_{\eta\sigma}(.,t)=0 &\text{on }\partial\Omega,
\end{cases}
\end{equation*}
and we extend it by $u_{\eta,s}(.,t)-u_{\eta}(.,t)$ in $\Omega_{s,\sigma}^\eta(t)$.\\\\
We claim that
\begin{equation}\label{limite_l2_ueta}
\displaystyle\lim_{\sigma\to 0} \displaystyle\lim_{\eta\to 0} \Vert e_{\eta\sigma}(.,t) \Vert_{L^2(\Omega)}=0  \text{ a.e. on }]0,T[.
\end{equation}
From lemma \ref{inclusion_omega_s}, for a given $\sigma>0$, there exists $\eta_0>0$ such that $\forall \eta<\eta_0$,
\begin{equation*}
\oes \subset \Omega_{s,\sigma}(t)\text{ and }\os\subset\Omega_{s,\sigma}^\eta(t).
\end{equation*}
Let $\eta<\eta_0$. We write
\begin{equation}
{\Vert e_{\eta\sigma}(.,t) \Vert}^2_{L^2(\Omega)}={\Vert e_{\eta\sigma}(.,t) \Vert}^2_{L^2(\oes)}+{\Vert e_{\eta\sigma}(.,t) \Vert}^2_{L^2(\Omega_{s,\sigma}(t)\backslash\oesb)}+{\Vert e_{\eta\sigma}(.,t) \Vert}^2_{L^2(\Omega\backslash\overline{\Omega_{s,\sigma}}(t))}.\label{dec}
\end{equation}
From estimate (\ref{hu_bornee_1}), there holds
\begin{equation}
\int_0^T{\Vert e_{\eta\sigma}(.,t) \Vert}^2_{L^2(\oes)}dt \leq C \eta.\label{esti}
\end{equation}
Since $\Omega_{s,\sigma}(t)\backslash\oesb$ has width less than $2\sigma$, from the proof of lemma 5.10 of \cite{Fujita70} we have a.e. on $]0,T[$,
\begin{eqnarray}
{\Vert e_{\eta\sigma}(.,t) \Vert}^2_{L^2(\Omega_{s,\sigma}(t)\backslash\oesb)} & \leq & C \left({\Vert e_{\eta\sigma}(.,t) \Vert}^2_{L^2(\partial\oes)}+\sigma^2 {\Vert \nabla e_{\eta\sigma}(.,t) \Vert}^2_{L^2(\Omega_{s,\sigma}(t)\backslash\oesb)}\right).\nonumber\\
\end{eqnarray}
And using a trace theorem, we get
\begin{eqnarray}
{\Vert e_{\eta\sigma}(.,t) \Vert}^2_{L^2(\Omega_{s,\sigma}(t)\backslash\oesb)} & \leq & C \left({\Vert e_{\eta\sigma}(.,t) \Vert}_{L^2(\oes)}{\Vert \nabla e_{\eta\sigma}(.,t) \Vert}_{L^2(\oes)}\right.\nonumber\\
&&\qquad\qquad\qquad\qquad\qquad\qquad\quad\left.+\sigma^2 {\Vert \nabla e_{\eta\sigma}(.,t) \Vert}^2_{L^2(\Omega_{s,\sigma}(t)\backslash\oesb)}\right)\nonumber\\
& \le & C \left({\Vert e_{\eta\sigma}(.,t) \Vert}_{L^2(\oes)}{\Vert \nabla e_{\eta\sigma}(.,t) \Vert}_{L^2(\Omega_{s,\sigma}(t))}\right.\nonumber\\
&&\qquad\qquad\qquad\qquad\qquad\qquad\quad\left.+\sigma^2 {\Vert \nabla e_{\eta\sigma}(.,t) \Vert}^2_{L^2(\Omega_{s,\sigma}(t))}\right).\label{bof}
\end{eqnarray}
Adding ${\Vert e_{\eta\sigma}(.,t) \Vert}^2_{L^2(\oes)}$ to this inequality gives
\begin{multline}
{\Vert e_{\eta\sigma}(.,t) \Vert}^2_{L^2(\Omega_{s,\sigma}(t))} \leq C \left({\Vert e_{\eta\sigma}(.,t) \Vert}^2_{L^2(\oes)}+{\Vert e_{\eta\sigma}(.,t) \Vert}_{L^2(\oes)}{\Vert \nabla e_{\eta\sigma}(.,t) \Vert}_{L^2(\Omega_{s,\sigma}(t))}\right.\\\left.+\sigma^2 {\Vert \nabla e_{\eta\sigma}(.,t) \Vert}^2_{L^2(\Omega_{s,\sigma}(t))}\right).\label{baf}
\end{multline}
For the last term in (\ref{dec}) we use Lemma \ref{stokes}:
\begin{equation}
{\Vert e_{\eta\sigma}(.,t) \Vert}^2_{L^2(\Omega\backslash\overline{\Omega_{s,\sigma}}(t))} \le C {\Vert e_{\eta\sigma}(.,t) \Vert}_{L^2(\Omega_{s,\sigma}(t))} {\Vert \nabla e_{\eta\sigma}(.,t) \Vert}_{L^2(\Omega_{s,\sigma}(t))}.\label{bif}
\end{equation}
Since $e_{\eta\sigma}(.,t)=u_{\eta,s}(.,t)-u_{\eta}(.,t) \text{ in }\Omega_{s,\sigma}(t)$
and $u_{\eta}$, $u_{\eta,s}$ are bounded in $L^2(0,T;H^1_0(\Omega))$, we have
\begin{equation}
\int_0^T{\Vert \nabla e_{\eta\sigma}(.,t) \Vert}^2_{L^2(\Omega_{s,\sigma}(t))} dt\leq C.\label{buf}
\end{equation}
With (\ref{esti}) and (\ref{buf}) we are now in position to estimate the integral over $[0,T]$ of (\ref{bof}-\ref{bif}). By Cauchy-Schwarz inequality:
$$\int_0^T{\Vert e_{\eta\sigma}(.,t) \Vert}^2_{L^2(\Omega_{s,\sigma}(t)\backslash\oesb)}dt\le C(\eta^\frac12+\sigma^2),$$
$$\int_0^T{\Vert e_{\eta\sigma}(.,t) \Vert}^2_{L^2(\Omega\backslash\overline{\Omega_{s,\sigma}}(t))}dt \le C\left(\int_0^T {\Vert e_{\eta\sigma}(.,t) \Vert}^2_{L^2(\Omega_{s,\sigma}(t))} dt\right)^\frac12\le C(\eta+\eta^\frac12+\sigma^2)^\frac12.$$
Therefore, for a fixed value of $\sigma$ we can  pass to the limit in $\eta$, and then pass to the limit in $\sigma$, to obtain
\begin{equation}
\displaystyle\lim_{\sigma\to 0}\lim_{\eta\to 0}\int_0^T{\Vert e_{\eta\sigma}(.,t) \Vert}^2_{L^2(\Omega)}dt =0.
\end{equation}
This strong convergence can be turned into an almost everywhere in $t$ convergence up to the extraction of a subsequence. The rest of the proof is adapted in a straightforward way from that of Lemma \ref{convergence_Pu_u}.

\section*{Acknowledgments} This work was supported by the French Ministry of Education through ANR grant 06-BLAN-0306.


{}

\end{document}